\documentclass[a4paper,reqno]{amsart}
\usepackage[utf8]{inputenc}
\usepackage{amsmath,amssymb,amsthm,amsfonts}
\usepackage{bbm}
\usepackage{euscript}
\usepackage{enumitem}
\usepackage{nicefrac}
\usepackage{mathtools}	

\theoremstyle{plain}
\newtheorem{theorem}{Theorem}
\newtheorem{lemma}[theorem]{Lemma}
\newtheorem{corollary}[theorem]{Corollary}

\newtheorem{proposition}[theorem]{Proposition}
\theoremstyle{definition}

\newtheorem{remark}[theorem]{Remark}
\newtheorem{example}[theorem]{Example}
\numberwithin{theorem}{section}
\numberwithin{equation}{section}

\newcommand{\mbR}{{\mathbb R}}
\newcommand{\mbN}{{\mathbb N}}
\newcommand{\mbQ}{{\mathbb Q}}
\newcommand{\mbZ}{{\mathbb Z}}
\newcommand{\cF}{{\mathcal F}}

\newcommand{\ind}{\mathbbm{1}}		

\newcommand{\drift}{a}
\newcommand{\sign}{\mathop{\rm sign}}

\newcommand{\ve}{\varepsilon}

\renewcommand{\Pr}{{\mathbb{P}}} 		
\newcommand{\Exp}{{\mathbb{E}}} 		
\newcommand{\Var}{\mathrm{Var}}		
\renewcommand{\leq}{\leqslant}

\renewcommand{\geq}{\geqslant}
\renewcommand{\ge}{\geqslant}
\DeclareMathOperator{\wlim}{wlim}
\newcommand{\from}{\colon}		
\newcommand{\Lip}{{\mathrm{Lip}}}
\newcommand{\nqquad}{\hspace{-2em}}

\title{The zero-noise limit of SDEs with \(L^\infty\) drift}

\author[U. S. Fjordholm]{Ulrik Skre Fjordholm}
\author[M. Musch]{Markus Musch}
\address{Department of Mathematics, University of Oslo, PO Box 1053 Blindern, N-316 Oslo, Norway}
\author[A. Pilipenko]{Andrey Pilipenko}
\address{Institute of Mathematics,  National Academy of Sciences of Ukraine, Tereshchenkivska str. 3, 01601, Kiev, Ukraine}

\begin{document}

\begin{abstract}
We study the zero-noise limit for autonomous, one-dimensional ordinary differential equations with discontinuous right-hand sides. Although the deterministic equation might have infinitely many solutions, we show, under rather general conditions, that the sequence of stochastically perturbed solutions converges to a unique distribution on classical solutions of the deterministic equation. We provide several tools for computing this limit distribution.
\end{abstract}

\maketitle

\section{Introduction}
 Consider a  scalar, autonomous ordinary differential equation (ODE) of the form
\begin{equation}\label{eq:ode}
 \begin{split}
  \frac{dX}{dt}(t) &= \drift(X(t)) \qquad \text{for } t > 0, \\
  X(0) &= 0
 \end{split}
\end{equation}
where \( \drift\from\mbR \rightarrow \mbR \) is Borel measurable. (The initial data $X(0)=0$ can be translated to an arbitrary point $x_0\in\mbR$, if needed.)
If the drift $a$ is non-smooth then uniqueness of solutions might fail --- this is the \emph{Peano phenomenon}. To distinguish physically reasonable solutions from non-physical ones, we add stochastic noise to the equation, with the aim of letting the noise go to zero. Thus, we consider a stochastic differential equation
\begin{equation}\label{eq:ode_pert} %
\begin{split}
dX_\ve(t) &= \drift(X_\ve(t)) dt + \ve dW(t), \\
X_\ve(0) &= 0.
\end{split}
\end{equation}
where \( W(t) \) is a one-dimensional Brownian motion on a given probability space \( (\Omega, \cF, \Pr )\), and \( \ve > 0 \). By the Zvonkin--Veretennikov theorem \cite{Veretennikov1981,Zvonkin1974}, equation \eqref{eq:ode_pert} has a unique  strong solution.

In this paper we consider the following problem:
\begin{quotation}
\emph{Identify the limit $\lim_{\ve\to0} X_\ve$, and show that it satisfies \eqref{eq:ode}.}
\end{quotation}
Somewhat informally, the challenges are:
\begin{itemize}
\item determining whether the sequence $\{X_\ve\}_\ve$ (or a subsequence) converges, and in what sense;
\item identifying the limit(s), either by a closed form expression or some defining property;
\item proving that the limit solves \eqref{eq:ode} by passing to the limit in the (possibly discontinuous) term $a(X_\ve)$.
\end{itemize}

The problem originated in the 1981 paper by Veretennikov \cite{Veretennikov1981b}, and was treated extensively in the 1982 paper by Bafico and Baldi \cite{BaficoBaldi1982}. Only little work has been done on this problem since then, despite its great interest. The original work of Bafico and Baldi dealt with the Peano phenomenon for an autonomous ordinary differential equation. They considered continuous drifts which are zero at some point and are non-Lipschitz continuous on at least one side of the origin. In their paper they show that the $\ve\to0$ limit of the probability measure that represents the solution of the stochastic equation is concentrated on at most two trajectories. Further, they compute explicitly some limit probability measures for specific drifts. Unfortunately, since the result of Bafico and Baldi relies on the direct computation of the solution of an elliptic PDE, it only works in one dimension. In one dimension this elliptic PDE reduces to a second-order boundary value problem for which an explicit solution can be computed. Therefore, there is little hope that this approach will also work in higher dimensions.

The only other work that is known to us dating back to the previous century is the paper by Mathieu from 1994 \cite{Mathieu1994}. In 2001 Grandinaru, Herrmann and Roynette published a paper \cite{GradinaruHerrmannRoynette2001} which showed some of the results of Bafico and Baldi using a large deviations approach. Herrmann did some more work on small-noise limits later on together with Tugaut \cite{HerrmannTugaut2010, HerrmannTugaut2012, HerrmannTugaut2014}.

Yet another approach to Bafico and Baldi's original problem was presented by Delarue and Flandoli in \cite{DelarueFlandoli2014}. They apply a careful argument based on exit times. Noteworthy it also works in arbitrary dimension but with a very specific right-hand side, in contrast to the original assumption of a general continuous function; see also Trevisian \cite{Trevisian13}.
We also point out the recent paper by Delarue and Maurelli  \cite{DelarueMaurelli2020}, where multidimensional gradient dynamics with H\"older type coefficients was perturbed by a small Wiener noise.

The 2008 paper by Buckdahn, Ouknine and Quincampoix \cite{BuckdahnOuknineQuincampoix2008} shows that the the zero noise limit is concentrated on the set of all Filippov solutions of \eqref{eq:ode}. Since this set is potentially very large, this result is of limited use to us.

Even less work was done for zero noise limits with respect to partial differential equations. To our best knowledge the only paper published so far is Attanasio and Flandoli's note on the linear transport equation \cite{AttanasioFlandoli2009}.

A new approach was proposed by Pilipenko and Proske when the drift in  \eqref{eq:ode} has H\"older-type asymptotics in a neighborhood of $x=0$ and the perturbation is a self-similar noise \cite{PilipenkoProske2015}. They used space-time scaling and reduce a solution of the small-noise problem to a study of long time behaviour of a stochastic differential equation with a {\it fixed} noise.  This approach  can be generalized to multidimensional case and multiplicative Levy-noise perturbations \cite{PilipenkoProske2018, KulikPilipenko2020, PavlyukevichPilipenko2020, PilipenkoProske2021}.


\subsection{Uniqueness of classical solutions}
If the drift $a=a(x)$ is continuous then the question of existence and uniqueness of solutions of \eqref{eq:ode} is well established. If $a$ is {continuous} then it's known since Peano that there always exists at least one solution (at least for small times). Binding \cite{Binding1979} found that the solution is unique {if and only if} $a$ satisfies a so-called Osgood condition at all zeros $x_0$ of $a$:
\begin{equation}\label{eq:osgood_cond}
	\int_{x_0-\delta}^{x_0} \frac{1}{a(z)\wedge0}\,dz= -\infty,\qquad
\int_{x_0}^{x_0+\delta} \frac{1}{a(z)\vee0}\,dz = +\infty
\end{equation}
for all $\delta\in(0,\delta_0)$ for some $\delta_0>0$. (Here and in the remainder we denote \(\alpha \wedge \beta\coloneqq\min(\alpha,\beta)\) and $\alpha\vee\beta\coloneqq\max(\alpha,\beta)$.) The unique solution starting at $x$ is then given by
\begin{equation}\label{eq:deterministicsolution}
X(t;x) = \begin{cases}
	x & \text{if } a(x)=0 \\
	A^{-1}(t) & \text{if } a(x)\neq0
\end{cases}
\end{equation}
(at least for small $t$), where $A(y)\coloneqq\int_{x}^y 1/\drift(z)\, dz$ and $A^{-1}$ is its inverse function.

If $a$ is discontinuous --- say, $a\in L^\infty(\mbR)$ --- then the question of existence and uniqueness is much more delicate. The paper \cite{Fjordholm2018} gives necessary and sufficient conditions for the uniqueness of \emph{Filippov solutions} of \eqref{eq:ode}. We remark here that the extension to Filippov solutions might lead to non-uniqueness, even when the classical solution is unique. To see this, let $E\subset\mbR$ be measure-dense, i.e.~a set for which both $U\cap E$ and $U\setminus E$ have positive Lebesgue measure for any nonempty, open set $U\subset\mbR$ (see \cite{Rud83} for the construction of such a set), and let $a=1+\ind_E$. Then \eqref{eq:deterministicsolution} is the unique classical solution for any starting point $x\in\mbR$, whereas any function satisfying $\frac{d}{dt}X(t)\in[1,2]$ for a.e.~$t>0$ will be a Filippov solution. We will show that even in cases such as this one, the stochastically perturbed solutions converge to the classical solution, and not just any Filippov solution, as was shown in  \cite{BuckdahnOuknineQuincampoix2008}.


\subsection{Main result}
We aim to prove that the distribution of solutions $X_\ve$ of \eqref{eq:ode_pert} converges to a distribution concentrated on either a single solution of the deterministic equation \eqref{eq:ode}, or two ``extremal'' solutions. Based on the discussion in the previous section, we can divide the argument into cases depending on whether $a$ is positive, negative or changes sign in a neighbourhood, and in each case, whether an Osgood-type condition such as \eqref{eq:osgood_cond} holds. The case of negative drift is clearly analogous to a positive drift, so we will merely state the results for negative drift, without proof.

Under the sole assumption $a\in L^\infty(\mbR)$, the sequence $\{X_\ve\}_\ve$ is weakly relatively compact in $C([0,T])$, for any $T>0$. (Indeed, by \eqref{eq:ode_pert}, $X_\ve-\ve W$ is uniformly Lipschitz, and $\ve W\overset{P}{\to}0$ as $\ve\to0$. See e.g.~\cite{Billingsley1999} for the full argument.) Hence, the problems are to characterize the distributional limit of any convergent subsequence, to determine whether the entire sequence converges (i.e., to determine whether the limit is unique), and to determine whether the sense of convergence can be strengthened.

Without loss of generality we will assume that the process starts at $x=0$. If $a(0)=0$ but $a$ does \textit{not} satisfy the Osgood condition \eqref{eq:osgood_cond} at $x=0$, then both
$\psi_-$ and $\psi_+$ are classical solutions
of \eqref{eq:ode} (along with infinitely many other solutions), where
\begin{equation}\label{eq:maximalsolutions}
\psi_\pm(t) \coloneqq A_\pm^{-1}(t), \qquad \text{where } A_\pm(x) \coloneqq \int_0^x \frac{1}{a(z)}\,dz \text{ for } x\in\mbR_\pm.
\end{equation}
Generally, the functions $\psi_\pm$ are defined in a neighborhood of 0. We have assumed that $a$ is bounded, so
$\psi_\pm$ cannot blow up in finite time, but they can reach  singular points $R_\pm$ where $A_\pm$ blow up. If $t_\pm\in(0,\infty]$ are the times when $\psi_\pm(t_\pm)=R_\pm$ then we set $\psi_\pm(t)\equiv R_\pm$ for all $t\geq t_\pm$. We aim to prove that the distribution of $X_\ve$ converges to a distribution concentrated on the two solutions $\psi_-,\ \psi_+$, and to determine the weighting of these two solutions.

\begin{theorem}\label{thm:ZeroNoisePositiveDrift111}
Let $a\in L^\infty(\mbR)$ {satisfy $a\geq 0$} a.e.~in $(-\delta_0, \delta_0)$ for some $\delta_0>0$, and
\begin{equation}\label{eq:osgoodOnesided}
\int_{0}^{\delta_0} \frac{1}{a(z)} dz<\infty.
\end{equation}
Then, for any $T>0$, $X_\ve$ converges uniformly in probability to $\psi_+$:
\begin{equation}\label{eq:C2}
\big\|X_\ve-\psi_+ \big\|_{C([0,T])} \overset{P} \to 0 \qquad\text{as } \ve\to0.
\end{equation}
An analogous result holds for \emph{negative} drifts, with obvious modifications.
\end{theorem}

The proof of Theorem \ref{thm:ZeroNoisePositiveDrift111} for  strictly positive drifts $a$ is given in Section \ref{sec:positive_drift}, while the general case is considered in Section \ref{section:finalOfTheorem1.1}. The final theorem applies also to signed drifts:

\begin{theorem}\label{thm:ZeroNoiseRepulsive}
Let $a\in L^\infty(\mbR)$ satisfy
\begin{equation}\label{eq:osgoodrepulsive}
-\int_{\alpha}^{0} \frac{1}{a(z)\wedge0}\, dz<\infty, \qquad \int_{0}^{\beta} \frac{1}{a(z)\vee 0}\, dz<\infty
\end{equation}
for some $\alpha<0<\beta$ (compare with \eqref{eq:osgood_cond}). Let $\{\ve_k\}_k$ be some sequence satisfying $\ve_k>0$ and $\lim_{k\to\infty}\ve_k=0$, and define
\begin{equation}\label{eq:weights}
p_k \coloneqq
\frac{s_{\ve_k}(0)-s_{\ve_k}(\alpha)}{s_{\ve_k}(\beta)- s_{\ve_k}(\alpha)}
\in [0,1], \qquad
s_\ve(r) \coloneqq \int_0^r \exp\Bigl(-\frac{2}{\ve^2}\int_0^z a(u)\,du\Bigr)\,dz.
\end{equation}
Then $\{P_{\ve_k}\}_k$ is weakly convergent if $\{p_k\}_k$ converges. Defining $p\coloneqq \lim_{k}p_k$ and $P\coloneqq\wlim_k P_{\nu_k}$, we have
\begin{equation}\label{eq:limitMeasure}
P = (1-p)\delta_{\psi_-} + p\delta_{\psi_+}.
\end{equation}
\end{theorem}
The proof is given in Section \ref{sec:repulsive}, where we also provide tools for computing $p$.

\subsection{Outline of the paper}
We now give an outline of the rest of this manuscript. In Section \ref{sec:technical_results} we
give several technical results on convergence of SDEs with respect to perturbations of the drift;
 the relation between the solution and its
 exit time; and the distribution of the solution of an SDE.
 The goal of Section \ref{sec:positive_drift} is to prove Theorem \ref{thm:ZeroNoisePositiveDrift111} in the case where $a>0$, and in Section \ref{section:finalOfTheorem1.1} we extend to the case $a\geq0$. In Section \ref{sec:repulsive} we prove Theorem \ref{thm:ZeroNoiseRepulsive} and provide several results on sufficient conditions for convergence. Finally, we give some examples in Section \ref{sec:examples}.

\section{Technical results}\label{sec:technical_results}
In this section we list a few technical results. The first two results are comparison principles.  In order to prove them we use approximations by SDEs with smooth coefficients and the classical results on comparison. Since we do not suppose that the drift is smooth or even continuous, the results are not standard.

\begin{theorem}\label{thm:convergenceSDE_Thm}
Let $\{\drift_n\from \mbR \rightarrow \mbR \}_{n\geq0}$ be uniformly bounded
 measurable functions such that $\drift_n \to \drift_0$ pointwise a.e.~as $n\to\infty$. Let $X_n$ be a solution to the SDE
\[
X_n (t )= x_n + \int_0^t \drift_n (X_n (s )) ds + W(t),\qquad t\in[0,T].
\]
Then $\{X_n\}_n$ converges uniformly in probability:
\[
\bigl\|X_n(t)-X_0(t)\bigr\|_{C([0,T])} \overset{P}\to 0 \qquad \text{as } n\to\infty.
\]
\end{theorem}
For a proof, see e.g.~\cite[Theorem~2.1]{Pilipenko2013}.

\begin{theorem}\label{thm:comparisonThm}
Let \( \drift_1, \drift_2\from \mbR \rightarrow \mbR \) be locally bounded measurable functions satisfying \( \drift_1 \leq \drift_2\) and let $x_1\leq x_2$. Let \( X_1, X_2 \) be solutions to the equations
\begin{align*}
X_i (t )= x_i + \int_0^t \drift_i (X_i (s)) ds + W(t), \qquad i=1,2.
\end{align*}
Then
\[ X_1 (t )\leq X_2 (t )\qquad \forall\ t \geq 0 \]
with probability 1.
\end{theorem}
The proof is given in Appendix \ref{app:comparisonprinciple}.

\begin{lemma}\label{lem:timeinversion}
Let $\{f_n\}_{n\geq 1}\subset C([0,T])$ be a uniformly convergent sequence of non-random continuous functions and let $f_0\in C([0,T])$ be a strictly increasing function. Set $\tau^x_n\coloneqq\inf\bigl\{t\geq 0 : f_n(t)=x\bigr\}$ for every $n\geq 0$, and assume that
\[
\tau^x_n \to\tau^x_0 \qquad \text{for every } x\in \big(f_0(0), f_0(T)\bigr)\cap\mbQ.
\]
Then
\[
f_n\to f_0 \qquad \text{in } C([0,T]) \text{ as } n\to\infty.
\]
\end{lemma}
\begin{proof}
Let $\mathcal{T}\coloneqq f_0^{-1}(\mbQ)$, and note that this is a dense subset of $[0,T]$, since $f_0^{-1}$ is continuous. Let $t\in\mathcal{T}$ be arbitrary and let $x\coloneqq f_0(t)\in\mbQ$. By assumptions of the lemma we have $t=\tau_0^x=\lim_{n\to\infty}\tau_n^x.$
Moreover, since $f_n(\tau^x_n)=x$ for sufficiently large  $n$, we have
\begin{equation}\label{eq:240}
f_0(t)=x=\lim_{n\to\infty}f_n(\tau^x_n) = \lim_{n\to\infty} f_n(t),
\end{equation}
the last step following from the fact that $f_n$ converges uniformly and $\tau^x_n\to \tau^x_0=t$ as $n\to\infty$. Thus, $\{f_n\}_n$ converges pointwise to $f_0$ on a dense subset of $[0,T]$. But $\{f_n\}_n$ is uniformly convergent by assumption, so necessarily $f_n\to f_0$ uniformly.
\end{proof}
\begin{corollary}\label{cor:ConvergenceOfPaths}
Let $\{\xi_n\}_{n\geq 1} $ be a sequence of continuous stochastic processes $\xi_n\from[0,\infty)\to\mbR$
that is locally uniformly convergent with probability $1$. Let $\xi_0$ be a strictly increasing continuous process satisfying $\xi_0(0)=0$ and $\lim_{t\to\infty}\xi_0(t)=\infty$. Set $\tau_n^x\coloneqq\inf\{t\geq 0 :  \xi_n(t)\geq x\}$ and assume that
\[
\tau_n^x \overset{P}\to\tau_0^x \qquad \text{for every } x\in[0,\infty)\cap\mbQ.
\]
Then
\[
\xi_n \to \xi_0 \qquad \text{locally uniformly with probability }1.
\]
\end{corollary}
\begin{proof}
Enumerate the positive rational numbers as $\mbQ\cap (0,\infty)=\{x_n\}_n$.
Select a sequence $\{n^1_k\}_k$ such that
\[
\lim_{k\to\infty}\tau^{x_1}_{n^1_k} = \tau^{x_1}_0 \qquad \text{$\Pr$-a.s.}
\]
Then select a sub-subsequence $\{n^2_k\}_k$ of $\{n^1_k\}_k$ such that
\[
\lim_{k\to\infty}\tau^{x_2}_{n^2_k} = \tau^{x_2}_0 \qquad \text{$\Pr$-a.s.,}
\]
and so on. Then
\[
\Pr\Bigl(\forall\ j\in\mbN \quad \lim_{k\to\infty}\tau^{x_j}_{n^k_k} = \tau^{x_j}_0 \Bigr) = 1.
\]
From Lemma \ref{lem:timeinversion} it follows that
\[
\Pr\Bigl(\lim_{k\to\infty}\xi_{n^k_k}=\xi_0 \quad \text{uniformly in }[0,T]\Bigr)=1
\]
for any $T>0$. This yields the result.
\end{proof}

Assume that  $\drift, \sigma\from \mbR\to\mbR$ are bounded measurable functions, $\sigma$
 is separated from zero.
It is well known that the stochastic differential equation
\[
d\xi(t) = \drift(\xi(t))dt+ \sigma(\xi(t)) dW(t), \qquad t\geq 0,
\]
has a unique (weak) solution, which is a continuous strong Markov process, i.e., $\xi$ is a diffusion process.

Denote
 $L\coloneqq\drift(x)\frac{d}{dx}+\frac{1}2\sigma^2(x) \frac{d^2}{dx^2}$
 and
let $s$ and $m$ be a scale function and a speed measure of $\xi,$
see details in \cite[Chapter VII]{RevuzYor1999}.
Define the hitting time of $\xi$ as $\tau^y\coloneqq\inf\{t\geq 0 : \xi(t) =y\}$.
Recall that $s$ and $m$ are well-defined up to constants,
and $s$ is a non-degenerate $L$-harmonic function, i.e.,
\begin{equation}\label{eq:Lharmonic}
L s=0,
\end{equation}
in particular
\begin{equation}\label{eq:eq_scale}
s(x)\coloneqq\int_{y_1}^x\exp\left(-\int_{y_2}^y\frac{2 a(z)}{\sigma(z)^2}dz\right) dy,
\end{equation}
and
\begin{equation}\label{eq:463}
m(dy)=\frac{2}{s'(y)\sigma(y)^2}dy
\end{equation}
for any choices of $y_1, y_2,$ see \cite[Chapter VII, Exercise 3.20]{RevuzYor1999}.

\begin{theorem}\label{thm:exit_time} Let $x_1<x_2$ be arbitrary.
\begin{enumerate}[leftmargin=*,label=(\roman*)]
\item  \cite[Chapter VII, Proposition 3.2 and Exercise 3.20]{RevuzYor1999}
\label{thm:exit_time1}
\begin{align*}
\Pr^{x}\big(\tau^{x_1}\wedge \tau^{x_2}<\infty\big)=1 \qquad &\forall\ x\in[x_1,x_2] \\
\intertext{and}
\Pr^{x}\bigl(\tau^{x_1}< \tau^{x_2}\bigr)=\frac{s(x_2)-s(x)}{s(x_2)-s(x_1)} \qquad &\forall\ x\in[x_1,x_2],
\end{align*}
%
%
%
\item  \label{thm:exit_time3}\cite[Chapter VII, Corollary 3.8]{RevuzYor1999}
For any $I=(x_1,x_2) $, $x\in I$ and for any non-negative measurable function $f$ we have
\begin{equation}\label{eq:194}
\Exp^x\biggl(\int_0^{\tau^{x_1}\wedge \tau^{x_2}} \!\!f(\xi(t)) dt\biggr) =
\int_{x_1}^{x_2} \!G(x,y) f(y) m(dy),
\end{equation}
where   $G=G_I$ is a symmetric function such that
\[
G_I(x,y)=\frac{(s(x)-s(x_1))(s(x_2)-s(y))}{s(x_2)-s(x_1)}, \qquad x_1\leq x\leq y\leq x_2.
\]
\end{enumerate}
\end{theorem}

\begin{remark}\label{rem:harmonic_functions}~
\begin{enumerate}[leftmargin=*,label=(\textit{\roman*})]
\item The function $\tilde u(x)\coloneqq\Exp^x\Bigl(\int_0^{\tau^{x_1}\wedge \tau^{x_2}} f(\xi(t)) dt\Bigr)$ from the left-hand side of \eqref{eq:194}
is a solution to
\[
\begin{cases}
	L \tilde u(x) =-f(x), & x\in(x_1,x_2)\\
	\tilde u(x_1)=\tilde u(x_2)=0.
\end{cases}
\]
The function $G$ from  \eqref{eq:194} is the corresponding Green function, in the sense that $\tilde{u}(x)$ can be written as the right-hand side of \eqref{eq:194}.

\item \label{thm:exit_time2}
 If we take $f(x)=1$ in \eqref{eq:194}, then we get a formula for
  the expectation of the exit time  $u(x)\coloneqq\Exp^x(\tau^{x_1}\wedge \tau^{x_2})$, $x\in[x_1,x_2]$.
In particular,
\[u(x)=-\int_{x_1}^x2\Phi(y)\int_{x_1}^y \frac{dz}{\sigma(z)^2\Phi(z)}dy+
\int_{x_1}^{x_2}2\Phi(y)\int_{x_1}^y \frac{dz}{\sigma(z)^2\Phi(z)}dy \frac{\int_{x_1}^{x}\Phi(y)dy}{\int_{x_1}^{x_2}\Phi(y)dy},\]
where $\Phi(x)=\exp\left(-\int_{x_1}^x\frac{2 \drift(z)}{\sigma(z)^2}dz\right).$
\end{enumerate}
\end{remark}

Finally, the following result will be quite useful when taking limits
{$\sigma=\sigma_\ve(x)\coloneqq\ve\to0$} in terms such as $s$ and $u$ above.

\begin{lemma}\label{lem:approxidentity}
Let $\alpha<\beta$ and $\ve\neq0$, let $f,g\in L^1((\alpha,\beta))$ with $f>0$ almost everywhere, and let
\begin{equation*}
g_\ve(y)\coloneqq\int_{y}^{\beta}\exp\left(-\int_{y}^z \frac{f(u)}{\ve^2}\,du\right)\frac{f(z)}{\ve^2}g(z)\,dz, \qquad y\in[\alpha,\beta].
\end{equation*}
Then $g_\ve \to g$ as $\ve\to 0$ in $L^1((\alpha,\beta))$ and pointwise a.e.~ in $y\in(\alpha,\beta)$. The same is true if
\begin{equation*}
g_\ve(y)\coloneqq\int_{\alpha}^{y}\exp\left(-\int_z^y \frac{f(u)}{\ve^2}\,du\right)\frac{f(z)}{\ve^2}g(z)\,dz, \qquad y\in[\alpha,\beta].
\end{equation*}
\end{lemma}

The proof is given in Appendix \ref{app:comparisonprinciple}. Note that this lemma provides a positive answer to the
 question raised by Bafico and Baldi in \cite[Remark~b~in~Section~6]{BaficoBaldi1982} on whether
\cite[Proposition 3.3]{BaficoBaldi1982} still holds under the sole assumption of \( \int_0^r 1/a(z)dz < + \infty \).

\section{Positive drifts}\label{sec:positive_drift}
This section is dedicated to the proof of Theorem \ref{thm:ZeroNoisePositiveDrift111}. In order to prove the theorem, we first prove the following:
\begin{theorem}\label{thm:ZeroNoiseUnifPositive}
	Let $a\in L^\infty(\mbR)$ and assume that there exist positive constants $\delta_0,c_->0$
	such that
	\begin{equation}\label{eq:assumption_c_pm}
	a(x)\geq c_- \quad \text{for a.e. } x\in(-\delta_0,\infty).
	\end{equation}
	Then we have the uniform convergence in probability
	\begin{equation}\label{eq:result}
		\|X_\ve- \psi_+\|_{C([0,T])}\overset{P}\to 0 \quad \text{as } \ve\to0 \text{ for all }T>0.
	\end{equation}
\end{theorem}

\begin{proof}[Proof of Theorem \ref{thm:ZeroNoiseUnifPositive}]
The proof consists of these steps:
\begin{enumerate}[label=\arabic*.]
\item Show weak relative compactness of $\{X_\ve\}_\ve$.
\item Show that $\bar X_0$ is strictly increasing, where $\bar X_0$ is a limit point of $\{X_\ve\}_\ve$.
\item Reduce to proving convergence of the hitting
 times $\tau^\ve\to\tau$, see Lemma \ref{lem:timeinversion}.
\end{enumerate}

\medskip\noindent \textit{Step 1:}
For any $T>0$ the family $\{X_\ve\}_{\ve\in (0,1]}$ is weakly relatively compact in $C([0,T])$
(see e.g.~\cite{Billingsley1999}).
Since $\psi_+$ is non-random, the convergence statement \eqref{eq:result} is equivalent to the weak convergence
\[
X_\ve\Rightarrow  \psi_+ \qquad \text{ in } C([0,T]) \text{ as } \ve\to0 .
\]
for any $T>0$. To prove the latter, it suffices to verify that if $\{X_{\ve_k}\}_k$ is any convergent subsequence, then $\psi_+$ is its limit.

\medskip \noindent \textit{Step 2:}
Assume that $X_{\ve_k}\Rightarrow \bar X_0$ as $k\to\infty$. Since
\[
X_{\ve_k}(t)=\int_0^t \drift(X_{\ve_k}(s))\, ds+\ve_k W(t) \qquad \forall\ t\in[0,T],
\]
and $\ve_k W \overset{P}{\to} 0$, Slutsky's theorem implies that also
\begin{equation}\label{eq:Lip}
\int_0^\cdot \drift(X_{\ve_k}(s))\, ds \Rightarrow \bar X_0 \qquad \text{in }C([0,T]).
\end{equation}
By Skorokhod's representation theorem \cite[Theorem 1.6.7]{Billingsley1999}, we may assume that the convergence in \eqref{eq:Lip} happens almost surely. Since $c_-\leq a \leq c_+$ (for some $c_+>0$), we conclude that
\[
c_-\leq \frac{\bar X_0(t_2)-\bar X_0(t_1)}{t_2-t_1} \leq c_+ \qquad \forall\ t_1,t_2\in[0,T],  \text{ almost surely.}
\]
In particular, $\bar{X}_0$ is strictly increasing.

\medbreak \noindent \textit{Step 3:}
Notice that assumption \eqref{eq:assumption_c_pm} implies that $\lim_{t\to\infty}\psi_+(t)=+\infty.$
Define
\[
\tau_\ve^x\coloneqq\inf\{t\geq 0\,:\, X_\ve(t)=x\}, \qquad \tau_0^x \coloneqq \inf\{t\geq 0 \,:\, \psi_+(t)=x\} = A(x)
\]
where $A(x)\coloneqq \int_0^x a(z)^{-1}\,dz$ (cf.~\eqref{eq:deterministicsolution}).
By Corollary \ref{cor:ConvergenceOfPaths} it is enough to show convergence in probability of $\tau_\ve$:
\begin{equation}\label{eq:conv_hitting}
\tau_\ve^x\overset{P}\to A(x) \qquad\text{as }\ve\to0 \text{ for every } x\in\mbQ\cap [0,\infty).
\end{equation}
To check \eqref{eq:conv_hitting} it is sufficient to verify that
\begin{subequations}
\begin{alignat}{2}
&\lim_{\ve\to0} \Exp(\tau_\ve^x) = A(x) &\qquad&\text{for any } x\in\mbQ\cap [0,\infty), \label{eq:conv_hitting_expectation} \\
&\lim_{\ve\to0}\Var(\tau_\ve^x)= 0 &&\text{for any } x\in\mbQ\cap [0,\infty). \label{eq:conv_hitting_variance}
\end{alignat}
\end{subequations}
We prove these properties under less restrictive conditions on $a$, given in the lemma below.

\begin{lemma}\label{lem:properties_of_time}
Let $R,\delta>0$ and let $a\in L^\infty(\mbR)$ satisfy $a > 0$ a.e.~in $(-\delta,R)$. Assume that the Osgood-type condition
\begin{equation}\label{eq:positivedriftcondition}
	\int_{0}^R \frac{1}{a(z)}\, dz<\infty
\end{equation}
is satisfied. Denote $A(r)\coloneqq\int_0^r a(z)^{-1}\,dz$ for $r\in[0,R]$. Then
\begin{subequations}
\begin{alignat}{2}
&\lim_{\ve\to0}\Pr^x\big(\tau^{-\delta}_\ve>\tau^{R}_\ve\big)=1 &&\forall \ 0\leq x\leq R, \label{eq:ProbabilityFirstExit} \\
&\lim_{\ve\to0}\Exp^x\big(\tau^{-\delta}_\ve\wedge \tau^r_\ve\big) = A(r)
{-A(x)} &\qquad& \forall\ 0\leq x<r\leq R. \label{eq:ExpectedTrajectory} \\
\intertext{{Moreover, if $a(x)\geq c_-$ for $x\in(-\infty,-\delta)$ for some constant $c_->0$, then also}}
& {\lim_{\ve\to0}\Exp^0 ( \tau^r_\ve) =A(r)} &&\forall\ 0<r\leq R, \label{eq:ConvergenceOfExpectationsExits} \\
\intertext{ and if $a(x)\geq c_->0$ for all $ x\in\mbR$, then}
&{\lim_{\ve\to0}\Var^0(  \tau^r_\ve) =0} &&\forall\ 0<r\leq R. \label{eq:VanishingVariance}
\end{alignat}
\end{subequations}
\end{lemma}

We finalize the proof of Theorem \ref{thm:ZeroNoiseUnifPositive} and then prove the claims of Lemma \ref{lem:properties_of_time}  separately. Define the function
\[
\tilde a(x):=\begin{cases} a(x) & \text{if } x>-\delta,\\ c_- & \text{if } x\leq -\delta,\end{cases}
\]
and denote the solution to the corresponding stochastic differential equation by $\tilde X_\ve$.  It follows from Lemma \ref{lem:properties_of_time} that
\[
\|\tilde X_\ve- \psi_+\|_{C([0,T])}\overset{P}\to 0 \qquad \text{as } \ve\to0 \text{ for all }T>0.
\]
Uniqueness of the solution yields $\Pr\bigl(\tilde X_\ve(t)=  X_\ve(t) \text{ for } t\leq \tau_\ve^{-\delta}\bigr)=1.$
It is easy to see that $\Pr(\tau_\ve^{-\delta}=\infty)\to1 $ as $\ve\to0.$ This completes the proof of Theorem  \ref{thm:ZeroNoiseUnifPositive}.
\end{proof}

\begin{proof}[Proof of \eqref{eq:ProbabilityFirstExit} in Lemma \ref{lem:properties_of_time}]
By Theorem \ref{thm:exit_time}\ref{thm:exit_time1}, we can write
\[
\Pr^x(\tau^r_\ve<\tau^{-\delta}_\ve) = \frac{s_\ve(x)}{s_\ve(r)} \geq \frac{s_\ve(0)}{s_\ve(r)}
\]
for every $x\in[0,r]$, where (cf.~\eqref{eq:eq_scale}) 
\begin{equation}\label{eq:scalefunction}
s_\ve(x)\coloneqq\int_{-\delta}^xe^{-B(y)/\ve^2}\, dy, \qquad B(y) \coloneqq 2\int_{-\delta}^y a(z) dz.
\end{equation}
We have
\begin{equation}\label{eq:scale-function-estimate}
s_\ve(0) = \int_{-\delta}^0 e^{-B(y)/\ve^2}\,dy \geq \delta e^{-B(0)/\ve^2}
\end{equation}
since $B$ is nondecreasing. For sufficiently small $\ve>0$ we can find $y_\ve>0$ such that $B(y_\ve)=B(0)+\ve$. Note that $y_\ve\to0$ as $\ve\to0$. Again using the fact that $B$ is nondecreasing, we can estimate
\begin{align*}
s_\ve(r) &= s_\ve(0)+\int_0^r e^{-B(y)/\ve^2}\,dy
\leq s_\ve(0) + y_\ve e^{-B(0)/\ve^2} + (r-y_\ve)e^{-B(y_\ve)/\ve^2} \\
&\leq e^{-B(0)/\ve^2}\Bigl(s_\ve(0) + y_\ve + re^{-1/\ve}\Bigr).
\end{align*}
Using \eqref{eq:scale-function-estimate}, we get
\[
\Pr^x(\tau^r_\ve<\tau^{-\delta}_\ve) \geq \frac{s_\ve(0)e^{B(0)/\ve^2}}{s_\ve(0)e^{B(0)/\ve^2} + y_\ve+re^{-1/\ve}}
\geq \frac{\delta}{\delta + y_\ve+re^{-1/\ve}}.
\]
Since $y_\ve+re^{-1/\ve}\to0$ as $\ve\to0$, we conclude that $\Pr^x(\tau^r_\ve<\tau^{-\delta}_\ve)\to1$ as $\ve\to0$.
\end{proof}

\begin{proof}[Proof of \eqref{eq:ExpectedTrajectory} in Lemma \ref{lem:properties_of_time}]
We will show that for any $r\in(0,R]$ and $x\in[0,r]$, we have
$\lim_{\ve\to0} \Exp^x\big(\tau^{-\delta}_\ve \wedge \tau^r_\ve\big) = \int_x^r\drift(z)^{-1}dz.$
It follows from Theorem \ref{thm:exit_time}\ref{thm:exit_time3}\ with $x_1=-\delta$, $x_2=r$, $f\equiv1$,
 $s= s_\ve$ (cf.~\eqref{eq:scalefunction}) and $m=m_\ve$ (cf.~\eqref{eq:463}) that for any $\delta>0$ and $x\in[0,r]$,
 \begin{equation}\label{eq:668}
\begin{aligned}
&\Exp^x\big(\tau^{-\delta}_\ve \wedge \tau^{r}_\ve\big)
= \int_{-\delta}^r G_\ve(x,y)\,m_\ve(dy) \\
&= \int_{-\delta}^x G_\ve(y,x)\,m_\ve(dy)+\int_x^r G_\ve(x,y)\,m_\ve(dy) \\
&= \int_{-\delta}^x \frac{s_\ve(y)(s_\ve(r)-s_\ve(x))}{s_\ve(r)}\,
m_\ve(dy)+ \int_x^r \frac{s_\ve(x)(s_\ve(r)-s_\ve(y))}{s_\ve(r)}\,m_\ve(dy) \\
&= \int_{-\delta}^x \underbrace{\frac{s_\ve(y)}{s_\ve(r)}}_{\eqqcolon\, p_\ve(y)} (s_\ve(r)-s_\ve(x))\,
m_\ve(dy) + \underbrace{\frac{s_\ve(x)}{s_\ve(r)}}_{=\,p_\ve(x)} \int_x^r (s_\ve(r)-s_\ve(y))\, m_\ve(dy) \\
&= \int_{-\delta}^x p_\ve(y)\left[
\int_x^r\exp\left(-\int_{-\delta}^z\frac{2 \drift(u)}{\ve^2}du\right) dz\right]
\frac{2}{\ve^2} \exp\left(\int_{-\delta}^y\frac{2 \drift(z)}{\ve^2}dz\right) dy \\
&\quad + p_\ve(x)\int_x^r\left[
\int_{y}^r\exp\left(-\int_{-\delta}^z\frac{2 \drift(u)}{\ve^2}du\right) dz \right]
\frac{2}{\ve^2} \exp\left(\int_{-\delta}^y\frac{2 \drift(z)}{\ve^2}dz\right) dy \\
&= \int_{-\delta}^xp_\ve(y) \int_x^r\exp\left(-\int_y^z\frac{2 \drift(u)}{\ve^2}du\right)\frac{2}{\ve^2} \,dz
 dy \\
&\quad + p_\ve(x)\int_x^r\int_{y}^r\exp\left(-\int_y^z\frac{2 \drift(u)}{\ve^2}du\right)\frac{2}{\ve^2} \,dzdy \\
&= {
\int_{-\delta}^xp_\ve(y) \int_y^r\exp\left(-\int_y^z\frac{2 \drift(u)}{\ve^2}du\right)
\frac{2 \drift(z)}{\ve^2} \frac{\ind_{(x,r)}(z)}{ \drift(z)} \,dz
 dy }\\
&\quad + p_\ve(x)\int_x^r\int_{y}^r\exp\left(-\int_y^z\frac{2 \drift(u)}{\ve^2}du\right)
\frac{2 \drift(z)}{\ve^2} \frac{1}{ \drift(z)} \,dzdy \\
&= I_\ve + \mathit{II}_\ve.
\end{aligned}
\end{equation}
By Theorem \ref{thm:exit_time}\ref{thm:exit_time1} we have $p_\ve(x) = \Pr^x(\tau_\ve^{-\delta}>\tau_\ve^r)$, and \eqref{eq:ProbabilityFirstExit} in Lemma \ref{lem:properties_of_time} implies that $\lim_{\ve\to0}p_\ve(x)=1$ for every $x\in[0,r]$. Letting $f(z)=2a(z)$ and $g(z) = \frac{1}{\drift(z)}\ind_{(x,r)}(z)$ for $z\in[0,r]$, we see that the $z$-integral in $\mathit{II}_\ve$ can be written as
\[
{\int_y^r\exp\left(-\int_y^z\frac{f(u)}{\ve^2}du\right)\frac{f(z)}{\ve^2}g(z) \,dz.}
\]
 Note that $f,g\in L^1([0,r])$, by \eqref{eq:positivedriftcondition}. Thus, we can apply Lemma \ref{lem:approxidentity} with $\alpha=0$, $\beta=r$ to get
\[
g_\ve(y)\coloneqq\int_y^r\exp\left(-\int_y^u\frac{2 \drift(z)}{\ve^2}dz\right)\frac{2}{\ve^2} \,du \to g(y)
\]
in $L^1([0,r])$ and pointwise a.e.\ as $\ve\to0$, so that
\[
\mathit{II}_\ve \to \int_x^r g(y)\,dy = \int_x^r\frac{1}{a(y)}\,dy.
\]
A similar manipulation will hold for $I_\ve$, with the same functions $f$ and $g$, yielding
\[
I_\ve \to \int_{-\delta}^x \frac{1}{a(y)}\ind_{(x,r)}(y)\,dy = 0.
\]
Putting these together gives
\[
\lim_{\ve\to0}\Exp^x\big(\tau^{-\delta}_\ve \wedge \tau^{r}_\ve\big)
= \lim_{\ve\to0} I_\ve+\mathit{II}_\ve = \int_x^r \frac{1}{a(y)}\,dy.
\]
This concludes the proof.
\end{proof}

\begin{proof}[Proof of \eqref{eq:ConvergenceOfExpectationsExits} in Lemma \ref{lem:properties_of_time}]
{For any $x\in[0,r)$, note that $\lim_{\delta\to+\infty} \Exp^x(\tau^{-\delta}_\ve\wedge \tau^r_\ve)=\Exp^x(\tau^r_\ve)$.
Using \eqref{eq:668} and the assumption $a\geq c_->0$ it is easy to obtain the uniform estimates for
expectations and  to see that $\lim_{\ve\to0} \Exp^0(\tau^r_\ve)= A(r).$}
\end{proof}

\begin{proof}[Proof of \eqref{eq:VanishingVariance} in Lemma \ref{lem:properties_of_time}]
Let $X_\ve$ solve \eqref{eq:ode_pert} and define $Y_\ve(t) = \ve^{-2}X_\ve(\ve^2t)$. Substitution into \eqref{eq:ode_pert} then gives
\begin{equation}\label{eq:scaledSDE}
Y_\ve(t) = \int_0^t \drift\big(\ve^2 Y_\ve(s)\big)\,ds + B(t)
\end{equation}
where $B(t)=\ve^{-1}W(\ve^2t)$ is another Brownian motion. Applying the same scaling to $\tau$, we see that if $\pi^n_\ve$ is the exit time of $Y_\ve$ from $(-\infty,n]$ then $\pi^n_\ve = \ve^{-2}\tau^{\ve^2n}_\ve$.
%
 To this end, fix $x>0$, let $n=\ve^{-2} x$ (assumed for simplicity to be an integer) and define the increments $\zeta^1_\ve=\pi^1_\ve$, $\zeta^2_\ve=\pi^2_\ve-\pi^1_\ve$, $\dots$, $\zeta^n_\ve = \pi^n_\ve-\pi^{n-1}_\ve$. The strong Markov property ensures that $\zeta^1_\ve,\dots,\zeta^n_\ve$ are independent random variables. Hence,
\begin{align*}
\Var(\tau^x_\ve) &= \ve^4\Var(\pi^n_\ve) = \ve^4\Var\Biggl(\sum_{k=1}^n\zeta^k_\ve\Biggr) \\
&= \ve^4\sum_{k=1}^n\Var(\zeta^k_\ve).
\end{align*}
Hence, if we can bound $\Var(\zeta^k_\ve)$ by a constant independent of $\ve$, then $\Var(\tau^x_\ve) \leq \ve^4Cn = C x \ve^2 \to 0$, and we are done. To this end, note first the naive estimate $\Var(\zeta^k_\ve)\leq \Exp((\zeta^k_\ve)^2)$. Next, we invoke the comparison principle Theorem \ref{thm:comparisonThm} between $Y_\ve$ and
\[
Z_\ve(t)\coloneqq\int_0^t c_-\,dt+B(t) = c_-t+B(t),
\]
yielding $Z_\ve(t)\leq Y_\ve(t)$ for all $t\geq0$, almost surely. Hence, $\pi^n_\ve \leq \tilde{\pi}^n_\ve$, where $\tilde{\pi}^n_\ve$ is the exit time of $Z_\ve$, and correspondingly, $\zeta^k_\ve\leq \tilde{\zeta}^k_\ve$ for $k=1,\dots,n$. Since $(\tilde{\zeta}^k_\ve)_{k=1}^n$ are identically distributed, we get
\[
\Exp\big((\zeta^k_\ve)^2\big) \leq \Exp\big((\tilde{\zeta}^k_\ve)^2\big) = \Exp\big((\tilde{\zeta}^1_\ve)^2\big) = \Exp\big((\tilde{\pi}^1_\ve)^2\big).
\]
To estimate the latter, we have (letting $p_t = \mathrm{Law}(B_t) = \frac{1}{\sqrt{2\pi t}}e^{-|\cdot|^2/(2t)}$)
\begin{align*}
\Pr\big(\tilde{\pi}^1_\ve > t\big) &= \Pr\big(\tilde{\pi}^1_\ve > t,\ c_-t+B_t<1\big) + \underbrace{\Pr\big(\tilde{\pi}^1_\ve > t,\ c_-t+B_t \ge 1\big)}_{=\;0} \\
&\leq \Pr\big(c_-t+B_t<1\big) = \Pr\big(B_t<1-c_-t\big) \\
&= \int_{-\infty}^{1-c_-t} \frac{1}{\sqrt{2\pi t}}\exp\biggl(-\frac{|x|^2}{2t}\biggr)\,dx \\
&= \frac{1}{\sqrt{2\pi}}\int_{-\infty}^{(1-c_-t)/\sqrt{t}} \exp\biggl(-\frac{|y|^2}{2}\biggr)\,dy.
\end{align*}
It follows that
\[ \Exp((\tilde{\pi}^1_\ve)^2) = \int_0^\infty 2t \Pr(\tilde{\pi}^1_\ve > t)\,dt
\leq \frac{1}{\sqrt{2\pi}}\int_0^\infty 2t\int_{-\infty}^{(1-c_-t)/\sqrt{t}} \exp\left(-\frac{|y|^2}{2}\right)\,dy\,dt < \infty, \]
and we are done.
\end{proof}

Using the above theorem and standard comparison principles, we extend the result
to drifts satisfying an Osgood-type condition:
\begin{lemma}\label{lem:ZeroNoiseOsgood}
	Let $a\in L^\infty(\mbR)$ satisfy $a>0$ a.e.~in $(-\delta_0,\infty)$ for some $\delta_0>0$. Assume that for all $R>0$,
	\[
	\int_{0}^R \frac{1}{a(z)} dz<\infty.
	\]
	Then, for any $T>0$, $X_\ve$ converges to $\psi_+$:
	\begin{equation}\label{eq:C22}
		\big\|X_\ve-\psi_+\big\|_{C([0,T])} \overset{P} \to 0 \qquad\text{as } \ve\to0 \text{ for all }  T>0
	\end{equation}
(where $\psi_+$ is the maximal solution \eqref{eq:maximalsolutions}).
\end{lemma}
\begin{proof}
As in the proof of Theorem \ref{thm:ZeroNoiseUnifPositive} we know that $\{X_\ve\}_\ve$
 is weakly relatively compact, so it has some weakly convergent subsequence $\{X_{\ve_k}\}_k$. Due to Skorokhod's representation theorem \cite[Theorem 1.6.7]{Billingsley1999} there exists a sequence of copies
$\tilde X_{\ve_k}$ of $X_{\ve_k}$   that satisfy the corresponding SDEs with Wiener processes $B_{\ve_k}$ and such that
$\{\tilde X_{\ve_k}\}_k$ converges almost surely to some continuous non-decreasing process $\tilde X$:
\begin{equation}\label{eq:conv_tilde}
\Pr\Bigl(\lim_{k\to\infty} \|\tilde X_{\ve_k}-\tilde X\|_{C([0,T])}=0 \quad \forall\ T>0\Bigr)=1.
\end{equation}
{The limit process is non-decreasing, so without loss of generality we may assume that function
$\drift$ is such that $\drift(x)=c_-$ for all $x\in(-\infty,-\delta_0),$  where $c_->0$ is a constant.}
Define \( \drift_n \coloneqq \drift + \nicefrac{1}{n} \),
{let $\tilde X_{n,\ve}$ be the corresponding stochastic
process and let $X_n$ denote the solution of the corresponding deterministic problem}. It holds for all \( n \in \mbN \) that \( \drift_n \geq \nicefrac{1}{n} \),
 thus the result above holds for \( \drift_n \).

Let $\pi^x$, $\pi^x_{\ve_k}$, $\pi^x_{n,\ve_k}$, $\tau^x_n$ and $\tau^x$ be the hitting
 times of $\tilde X$, $\tilde X_{\ve_k}$, $\tilde X_{n,\ve_k}$, $X_n$ and $\psi_+$, respectively. By the comparison principle Theorem \ref{thm:comparisonThm}, we know that
\begin{equation}\label{eq:ineq_limits1}
\tilde X_{n,\ve_k} \geq \tilde X_{\ve_k}, \qquad \text{or equivalently,} \qquad \pi^x_{n,\ve_k} \leq \pi^x_{\ve_k}
 \; \forall\ x
\end{equation}
{(cf.~Lemma~\ref{lem:timeinversion}).}
It follows 
from Theorem \ref{thm:ZeroNoiseUnifPositive} that $\tilde X_{n,\ve_k}\to X_n$ a.s.~as $k\to\infty$, which together with \eqref{eq:conv_tilde} and \eqref{eq:ineq_limits1} implies
\begin{equation}\label{eq:ineq_limits2}
X_n \geq \tilde X, \qquad\text{or equivalently,}\qquad \tau^x_{n} \leq \pi^x\;\forall\ x.
\end{equation}
The lower semi-continuity of a hitting time with respect to its process also implies that $\pi^x\leq \liminf_{k\to\infty} \pi^x_{\ve_k}$
a.s. for any $x\geq 0$. Hence, for any $x\geq 0$,
\begin{align*}
A(x)&=\lim_{n\to\infty}A_n(x) = \lim_{n\to\infty} \tau_n^x \leq \Exp(\pi^x) \\
 &\leq \Exp\Bigl(\liminf_{k\to\infty} \pi_{\ve_k}^x\Bigr)
 \leq \liminf_{k\to\infty} \Exp\bigl(\pi_{\ve_k}^x\bigr)
 = A(x),
\end{align*}
the last equality following from \eqref{eq:ExpectedTrajectory} in Lemma \ref{lem:properties_of_time}. Hence, $\Exp(\pi^x)=A(x)$ for all $x\geq0$, and since $\pi^x\geq\tau_n^x\to A(x)$ as $n\to\infty$, we conclude that $\pi^x=A(x)$ almost surely for every $x\geq0$,
so Corollary \ref{cor:ConvergenceOfPaths} implies that $\tilde X=A^{-1}=\psi_+$ almost surely. Since  $\psi_+$ is non-random, we have the uniform convergence in probability
\[
\Pr\biggl(\lim_{k\to\infty}\|X_{\ve_k}- \psi_+\|_{C([0,T])}=0 \quad\forall\ T>0\biggr)=1.
\]
And finally, since the limit $\psi_+$ is unique, we can conclude that the entire sequence $\{X_\ve\}_\ve$ converges.
\end{proof}

We are now ready to prove Theorem \ref{thm:ZeroNoisePositiveDrift111} under the additional condition that $a>0$ a.e.~in $(-\delta_0,0)$:
\begin{proof}[Proof of Theorem \ref{thm:ZeroNoisePositiveDrift111} for positive $a$]
The case when $
 \int_{0}^{R} \frac{dx}{a(x)\vee0}<\infty
$ for any $R>0$ (and hence, in particular, $a>0$ a.e.~in $(-\delta_0,\infty)$)
has been considered in Lemma \ref{lem:ZeroNoiseOsgood}. Thus, we can assume that there is some $R>0$ such that $a>0$ a.e.~on $(-\delta_0,R)$, and for any (small) $\delta>0$,
\begin{equation}\label{eq:osgoodblowup}
 \int_0^{R-\delta} \frac{dx}{a(x)}<\infty \quad\text{but}\quad
 \int_0^{R+\delta} \frac{dx}{a(x)\vee 0}=\infty.
\end{equation}
Recall that
\[
\psi_+(x)=
\begin{cases}
A^{-1}(x),& x\in[0,A(R)),\\
R, & x\geq A(R).
\end{cases}
\]
(Note that $A(R)$ may be equal to $\infty.$) The proof of the theorem consists of the following steps:
\begin{enumerate}[label=\arabic*.]
\item Prove the theorem for the stopped process $X_\ve(\cdot\wedge\tau^R_\ve)$
\item Prove the theorem for nonnegative drifts
\item Extend to possibly negative drifts.
\end{enumerate}

\noindent\textit{Step 1.}
Set $\widehat a_m(x)\coloneqq a(x)\ind_{x\leq R-\nicefrac{1}{m}}+\ind_{x>R-\nicefrac1m}$ for
$m\in\mbN$, and note that $\widehat a_m$ satisfies the conditions of
 Lemma \ref{lem:ZeroNoiseOsgood}. Let $\widehat{X}_{m,\ve} $ denote the solution to the
corresponding SDE, $\widehat{X}_{m} $ its limit, and $\widehat{\tau}_{m,\ve}^x,\ \widehat{\tau}_{m }^x$
the corresponding
hitting times. It follows from the uniqueness of a solution that
\[
\Pr\Bigl( \widehat{\tau}_{m,\ve}^{R-\nicefrac1m}=\widehat{\tau}_\ve^{R-\nicefrac1m}\Bigr)=1 \quad\text{and}\quad
\Pr\Bigl(\widehat{X}_{m,\ve}(t) = X_{\ve}(t) \quad\forall\ t\leq
\widehat{\tau}_\ve^{R-\nicefrac1m}\Bigr)=1.
\]
Thus, by Lemma \ref{lem:ZeroNoiseOsgood},
\begin{equation}\label{eq:605}
\begin{split}
\sup_{t\in[0,T]}\big|X_{\ve}\bigl(t\wedge \widehat{\tau}_\ve^{R-\nicefrac{1}{m}}\bigr)- A^{-1}\big(t\wedge \widehat{\tau}_\ve^{R-\nicefrac{1}{m}}\big)\big|
&\overset{P} \to 0 \qquad\text{as } \ve\to0 \text{ for all }  T>0,
\\
\sup_{t\in[0,T]}\big|\widehat{X}_{m,\ve}\bigl(t\wedge \widehat{\tau}_\ve^{R-\nicefrac{1}{m}}\bigr)- A^{-1}\bigl(t\wedge \widehat{\tau}_\ve^{R-\nicefrac1m}\bigr)\big|
&\overset{P} \to 0 \qquad\text{as } \ve\to0 \text{ for all }  T>0,
\end{split}
\end{equation}
for every $m\in\mbN$.

Let  $\overline X_0$ be a limit point of $\{X_\ve\}_\ve$ and
 $X_{\ve_k}\Rightarrow  \overline X_0$ as $k\to\infty.$
It follows from \eqref{eq:605} that $\overline X_0(\cdot\wedge \tau^{R-\nicefrac1m}_m) = A^{-1}(\cdot\wedge \tau^{R-\nicefrac1m}_m )$, and since $m$ is arbitrary, we have $\overline{X}_0(\cdot\wedge \tau^{R} ) = A^{-1}(\cdot\wedge \tau^{R} )$, that is, $\overline X_0(\cdot\wedge\tau^R) = \psi_+(\cdot\wedge\tau^R)$. In particular, the entire sequence of stopped processes converges, by uniqueness of the limit.

\medskip\noindent\textit{Step 2.}
Assume next, in addition to \eqref{eq:osgoodblowup}, that $a\geq0$ a.e.~in $\mbR$. Any limit point of $\{X_\ve\}_\ve$ is a non-decreasing process, so to prove the theorem it suffices to verify that for any $\delta>0$ and $M>0$
\[
\limsup_{k\to\infty}\Pr \bigl( \tau^{R+\delta}_{\ve_k}<M\bigr)=0
\]
Set $a_n\coloneqq a+\nicefrac{1}{n}$ and let $ X_{n,\ve}$ denote
 the solution to the corresponding SDE.
It follows from comparison Theorem \ref{thm:comparisonThm} that for any $M>0$
\[
\limsup_{k\to\infty}\Pr\bigl(\tau^{R+\delta}_{\ve_k}<M\bigr)\leq
\liminf_{n\to\infty}\limsup_{k\to\infty}\Pr\bigl(\tau^{R+\delta}_{n,\ve_k}<M\bigr).
\]
Theorem \ref{thm:ZeroNoiseUnifPositive} implies that $\lim_{\ve\to0} X_{n,\ve}=X_n=A^{-1}_n,$ so the right hand side of the above
inequality equals zero for any $M$. This concludes the proof if $a$ is non-negative everywhere.

\medskip\noindent\textit{Step 3.}
In the case that $a$ takes negative values, we consider the processes $X_\ve^+$ satisfying the corresponding SDEs with drift $a^+(x)\coloneqq a(x)\vee 0$. We have already proved in Step 2 that
\begin{alignat*}{2}
\bigl\|X_\ve^+-\psi_+\bigr\|_{C([0,T])} \overset{P}\to 0 && \text{as }\ve\to0 \;\forall\ T>0 \\
\intertext{(since $a^+$ has the same deterministic solution $\psi_+$ as $a$ does), and in Step 1 that}
\bigl\|X_\ve\big(\cdot\wedge \tau^R_0\big)-\psi_+\bigr\|_{C([0,T])} \overset{P}\to 0 &\qquad& \text{as }\ve\to0\;\forall\ T>0.
\end{alignat*}
Theorem \ref{thm:comparisonThm} yields $X_\ve^+(t)\geq  X_\ve(t)$. Therefore, any (subsequential) limit of $\{X_\ve^+\}_\ve$ is greater than or equal to a limit of $\{ X_\ve\}_\ve$, and if $\bar X_0$ is a limit point of $\{X_\ve\}_\ve$ then
\[
\Pr\Bigl(\bar X_0(t) = \psi_+(t) \ \forall\ t\leq\tau^R_0 \text{ and } \bar{X}_0(t) \leq R \ \forall\ t>\tau^R_0\Bigr) =1.
\]
On the other hand, it can be seen that any limit point $\bar X_0$ of $\{X_\ve\}_\ve$
satisfies
\[
\Pr\Bigl(\exists\ t\geq \tau^0_R : \bar X_0(t)<R\Bigr)=0.
\]
Thus we have equality, $\bar X_0(t)=\psi_+(t)$ for all $t\geq 0 $  almost surely. This concludes the proof for the case $a(x)>0$ for $x\in(-\delta_0,0)$. The case $a(x)\geq 0$ for $x\in(-\delta_0,0)$ will be considered in \S\ref{section:finalOfTheorem1.1}.
\end{proof}

\section{Velocity with a change in sign}\label{sec:repulsive}
In this section we consider the repulsive case and prove Theorem \ref{thm:ZeroNoiseRepulsive}. We also provide several tools for computing the zero noise probability distribution.

\subsection{Convergence in the repulsive case}

\begin{lemma}\label{lem:osgoodrepulsive}
Let $\alpha<0<\beta$, assume that $a\in L^\infty(\mbR)$ satisfies the ``repulsive Osgood condition'' \eqref{eq:osgoodrepulsive}, and define $p_\ve$ by
\begin{equation}\label{eq:weightdef}
p_\ve \coloneqq \frac{- s_\ve(\alpha)}{s_\ve(\beta)- s_\ve(\alpha)}, \qquad
s_\ve(r) \coloneqq \int_0^r e^{-B(z)/\ve^2} \,dz, \qquad B(z)\coloneqq 2\int_0^z a(u)\,du.
\end{equation}
Then
\[
\limsup_{\ve\to0}\Exp^0\big(\tau_{\ve}^\alpha\wedge \tau_{\ve}^\beta\big) \leq \int_\alpha^\beta \frac{1}{|a(x)|}\,dx < \infty.
\]
If $p_{\ve_k}\to p$ as $k\to\infty$, then
\[
\Exp^0\big(\tau_{\ve_k}^\alpha\wedge \tau_{\ve_k}^\beta\big) \to {(1-p)}\int_\alpha^0 \frac{-1}{a(z)}\,dz + {p}\int_0^\beta \frac{1}{a(z)}\,dz \qquad \text{as }k\to\infty.
\]
\end{lemma}
\begin{proof}
{By \eqref{eq:Lharmonic}, \eqref{eq:463}, and \eqref{eq:194}
 with $f=1$} we can write
{\begin{align*}
&\Exp^0\big(\tau_{\ve}^\alpha\wedge \tau_{\ve}^\beta\big)
= \int_\alpha^0 \frac{(s_\ve(y)-s_\ve(\alpha))(s_\ve(\beta)-s_\ve(0))}{s_\ve(\beta)-s_\ve(\alpha)}\frac{2e^{B(y)/\ve^2}}{\ve^2}\,dy \\
&\qquad +\int_0^\beta \frac{(s_\ve(0)-s_\ve(\alpha))(s_\ve(\beta)-s_\ve(y))}{s_\ve(\beta)-s_\ve(\alpha)}\frac{2e^{B(y)/\ve^2}}{\ve^2}\,dy \\
&\quad= {(1-p_\ve)} \int_\alpha^0 (s_\ve(y)-s_\ve(\alpha))\frac{2e^{B(y)/\ve^2}}{\ve^2}\,dy  + {p_\ve}\int_0^\beta (s_\ve(\beta)-s_\ve(y))\frac{2e^{B(y)/\ve^2}}{\ve^2}\,dy \\
&\quad= {(1-p_\ve) \int_\alpha^0\int_\alpha^y \frac{2e^{(B(y)-B(z))/\ve^2}}{\ve^2}\,dz\,dy+
p_\ve \int_0^\beta\int_y^\beta\frac{2e^{(B(y)-B(z))/\ve^2}}{\ve^2}\,dz\,dy}\\
&\quad= (1-p_\ve) \int_\alpha^0\int_\alpha^y \frac{2\exp\Bigl({\textstyle -\int_z^y \frac{2a(u)}{\ve^2} du}\Bigr)}{\ve^2}\,dz\,dy \\
&\qquad +p_\ve \int_0^\beta\int_y^\beta\frac{2\exp\Bigl({\textstyle -\int_z^y \frac{2a(u)}{\ve^2} du}\Bigr)}{\ve^2}\,dz\,dy\\
&\quad= (1-p_\ve) \int_\alpha^0\int_\alpha^y \exp\Bigl({\textstyle-\int_z^y \frac{2a(u)}{\ve^2} du}\Bigr) \frac{2 a(z)}{\ve^2}\frac{1}{a(z)}\,dz\,dy \\
&\qquad +p_\ve \int_0^\beta\int_y^\beta \exp\Bigl({\textstyle-\int_z^y \frac{2a(u)}{\ve^2} du}\Bigr) \frac{2 a(z)}{\ve^2}\frac{1}{a(z)}\,dz\,dy.
\end{align*}}
Setting $f(z)=2\sign(z)a(z)$ and $g(z)=\frac{1}{a(z)}$ in Lemma \ref{lem:approxidentity}, we find that the above two integrals with $\ve=\ve_k$ converge to
\[
\int_\alpha^0 \frac{-1}{a(z)}\,dz \qquad\text{and}\qquad \int_0^\beta\frac{1}{a(z)}\,dz
\]
respectively, as $k\to\infty$. This concludes the proof.
\end{proof}

We can now prove the main theorem in the repulsive case.
\begin{proof}[Proof of Theorem \ref{thm:ZeroNoiseRepulsive}]
Let $X_{\ve_k'}$ be any weakly convergent subsequence of $\{X_{\ve_k}\}_k$, and let $\tau_{\ve_k'}$ and $\tau$ be the hitting times of $X_{\ve_k'}$ and its limit, respectively.
By Lemma \ref{lem:osgoodrepulsive} we have for any $\alpha<0<\beta$
\[
\Exp^0(\tau^\alpha\wedge\tau^\beta)\leq \liminf_{k\to\infty}\Exp^0\bigl(\tau^\alpha_{\ve_k}\wedge\tau^\beta_{\ve_k}\bigr) = {(1-p)A(\alpha)+ pA(\beta)}.
\]
Consequently, $\Pr^0\bigl(\tau^\alpha\wedge\tau^\beta=\infty\bigr)=0$, so $\Pr^0(\tau^\alpha<\tau^\beta)=\lim_{k\to\infty}\Pr^0(\tau^\alpha_{\ve_k'}<\tau^\beta_{\ve_k'})={1-p}$ and $\Pr^0(\tau^\alpha>\tau^\beta)={p}$.
Using Theorem \ref{thm:ZeroNoisePositiveDrift111} and the strong Markov property, the probability of convergence once the process escapes $(\alpha,\beta)$ at $x=\beta$ is one:
\[
\lim_{k\to\infty}\Pr^0\Bigl(\bigl\|X_{\ve_k'}(\cdot-\tau_\beta)-\psi_+(\cdot-A(\beta))\bigr\|_{C([0,T])}>\delta \bigm| \tau^\alpha>\tau^\beta \Bigr) = 1,
\]
for any sufficiently small $\delta>0$, and likewise for those paths escaping at $x=\alpha$. Passing $\alpha,\beta\to0$ yields
\begin{align*}
&\lim_{\delta\to0}\lim_{k\to\infty}\Pr^0\Bigl(\|X_{\ve_k'}-\psi_-\|_{C([0,T])}>\delta\Bigr) = {1-p}, \\ &\lim_{\delta\to0}\lim_{k\to\infty}\Pr^0\Bigl(\|X_{\ve_k'}-\psi_+\|_{C([0,T])}>\delta\Bigr) = {p}.
\end{align*}
Since this is true for any weakly convergent subsequence $\ve_k'$, and the limit is unique, the entire sequence $\ve_k$ must converge.
\end{proof}

\subsection{Probabilities in the repulsive case}
{Theorem \ref{thm:ZeroNoiseRepulsive} gives a concrete condition for convergence of the sequence
 of perturbed solutions, as well as a characterization of the limit distribution. In this section we give
 an explicit expression for the probabilities in the limit distribution, and an equivalent condition for
convergence.}

Consider the integral
\[
B(x)\coloneqq \int_0^x a(y)\,dy
\]
%
and denote $B_\pm = B\bigr|_{\mbR_\pm}$.
{Select any $\alpha>0, \beta>0$ such that the function $\mu\from[0,\beta)\to(\alpha,0]$
 defined by $\mu=B_-^{-1}\circ B_+$ is well-defined --- that is,
 \[
B_+(x) = B_-(\mu(x)), \quad \forall\ x\in [0,\beta).
\]
Clearly, $B_\pm$ are Lipschitz continuous. Since $a$ is strictly positive (negative) for $x>0$ ($x<0$), the inverses of $B_\pm$ are absolutely continuous (see e.g.~\cite[Exercise 5.8.52]{Bogachev2007}), so $\mu$ is also absolutely continuous.
We now rewrite the probability of choosing the left/right extremal  solutions $X^\pm$ in terms of $\mu$.}

\begin{theorem}\label{thm:limitprobs}
Let $a\in L^\infty(\mbR)$ satisfy \eqref{eq:osgoodrepulsive}
and let $\mu\from[0,\beta)\to(\alpha,0]$ be as above.
 Then $\{p_\ve\}_\ve$ converges if either the derivative $\mu'(0)$ exists, or if $\mu'(0)=-\infty$. In either case, we have
\begin{subequations}\label{eq:limit_prob}
\begin{equation}\label{eq:limit_prob1}
\lim_{\ve\to0}p_\ve = {\frac{-\mu'(0)}{1-\mu'(0)}}.
\end{equation}
Moreover, the derivative $\mu'(0)$ exists if and only if the limit $\lim_{u\downarrow0}\frac{B_-^{-1}(u)}{B_+^{-1}(u)}$ exists, and we have the equality:
\begin{equation}
\label{eq:limit_prob2}
\mu'(0)=\lim_{u\downarrow0}\frac{B_-^{-1}(u)}{B_+^{-1}(u)}.
\end{equation}
\end{subequations}
\end{theorem}

To prove the theorem we will need the following lemmas:

\begin{lemma}\label{lem:limits}
Let $\alpha<0<\beta$. Define $p_\ve$ as in \eqref{eq:weightdef} and $p_\ve'$ similarly, where
$\alpha,\beta$ are exchanged with any
 $\alpha'<0<\beta'.$
Then $\lim_{\ve\to0}p_\ve'/p_\ve = 1$. In particular,
$p_{\ve_k}$ converges to some $p$ as $k\to\infty$ if and only if $p_{\ve_k}'$ converges to $p$.
\end{lemma}
The proof follows from the following observation: Since $B$ is strictly increasing,
 then for any positive $r_1< r_2$ or negative $r_1>r_2$,
\[
 \lim_{\ve\to0}\frac{\int_{r_1}^{r_2} e^{-B(z)/\ve^2} \,dz}{\int_0^{r_1} e^{-B(z)/\ve^2} \,dz}=0.
\]
{Next, we prove a technical lemma:}
\begin{lemma}\label{lem:approxidentity2}
Let $0<a \in L^\infty([0,\beta])$ and $f\in L^1(\mbR)$, and for $\ve>0$ and $x\in[0,\beta)$ define
\begin{gather*}
B(x) = 2\int_0^x a(y)\,dy, \qquad \nu_\ve(x) = e^{-B(x)/\ve^2}\ind_{[0,\beta]}(x), \\
\bar{\nu}_\ve = \int_0^\beta \nu_\ve(y)\,dy, \qquad
f_\ve(x) = \frac{1}{\bar{\nu}_\ve}\int_0^\beta f(x+y)\nu_\ve(y)\,dy.
\end{gather*}
Then $f_\ve(x) \to f(x)$ as $\ve\to0$ if and only if $x$ is a Lebesgue point of $f$.
\end{lemma}
\begin{proof}
{Let $x\in[0,\beta)$. For $s\in(0,\beta-x)$, let
\[
F(s) = \int_0^s |f(x+y)-f(x)|\,dy, \qquad C_s = \sup_{y\in(0,s)}\tfrac{F(y)}{y}.
\]
Then $C_s \to 0$ as $s\to0$ if and only if $x$ is a Lebesgue point.} We estimate
\begin{align*}
|f_\ve(x)-f(x)| &= \frac{1}{\bar\nu_\ve}\biggl|\int_0^\beta (f(x+y)-f(x))\nu_\ve(y)\,dy\biggr| \\
&\leq \underbrace{\frac{1}{\bar{\nu}_\ve}\int_0^s |f(x+y)-f(x)|\nu_\ve(y)\,dy}_{=\,I_1} + \underbrace{\frac{1}{\bar{\nu}_\ve}\int_s^\beta |f(x+y)-f(x)|\nu_\ve(y)\,dy}_{=\,I_2}.
\end{align*}
For the first term we integrate by parts several times to get
\begin{align*}
I_1 &= F(s)\frac{\nu_\ve(s)}{\bar\nu_\ve} - \frac{1}{\bar\nu_\ve}\int_0^s F(y)\nu_\ve'(y)\,dy
\leq F(s)\frac{\nu_\ve(s)}{\bar\nu_\ve} - \frac{C_s}{\bar\nu_\ve}\int_0^s y\nu_\ve'(y)\,dy \\
&= F(s)\frac{\nu_\ve(s)}{\bar\nu_\ve} - \frac{C_s}{\bar\nu_\ve} s\nu_\ve(s) + \frac{C_s}{\bar\nu_\ve}\int_0^s \nu_\ve(y)\,dy \\
&\leq F(s)\frac{\nu_\ve(s)}{\bar\nu_\ve} + \frac{C_s}{\bar\nu_\ve}\int_0^\beta \nu_\ve(y)\,dy\\
&= F(s)\frac{\nu_\ve(s)}{\bar\nu_\ve} + C_s.
\end{align*}
For the second term we estimate
\begin{align*}
I_2 &\leq 2\|f\|_{L^1}  \frac{\nu_\ve(s)}{\bar\nu_\ve}.
\end{align*}
If we can find $s=s_\ve$ such that both $s_\ve\to0$ and $\frac{\nu_\ve(s_\ve)}{\bar\nu_\ve} \to 0$
as $\ve\to0$, then both $I_1$ and $I_2$ vanish in the $\ve\to0$ limit,
and we can conclude the result. Below we explain the existence of such a choice.

Since $B$ is increasing and Lipschitz continuous, with $B(0)=0$ and $\|B\|_\Lip \leq 2\|a\|_{L^\infty}<\infty$, there is some $\kappa<s$ satisfying $B(\kappa)=\tfrac12 B(s)$, and $\kappa\geq \frac{1}{2\|B\|_\Lip}B(s)$. Moreover, since $\nu_\ve$ is decreasing we have
\[
\bar\nu_\ve = \int_0^\beta \nu_\ve(y)\,dy \geq \kappa\nu_\ve(\kappa) = \kappa e^{-B(\kappa)/\ve^2} = \kappa e^{-B(s)/(2\ve^2)},
\]
so
\[
\frac{\nu_\ve(s)}{\bar\nu_\ve} \leq \frac{1}{\kappa}e^{-B(s)/(2\ve^2)}
\leq 2\|B\|_\Lip \frac{e^{-B(s)/(2\ve^2)}}{B(s)}.
\]
Now choose $s=s_\ve$ such that $B(s_\ve) = \ve$. (Such a number exists for sufficiently small $\ve>0$.) Then $s_\ve\to0$ as $\ve\to0$, and
\[
\frac{\nu_\ve(s)}{\bar\nu_\ve} \leq 2\|B\|_\Lip \frac{e^{-1/(2\ve)}}{\ve} \to 0
\]
as $\ve\to0$. This finishes the proof.
\end{proof}

\begin{proof}[Proof of Theorem \ref{thm:limitprobs}]
We have
\[
p_\ve = \frac{{-s_\ve(\alpha)}}{s_\ve(\beta)-s_\ve(\alpha)} =
{\frac{-\frac{s_\ve(\alpha)}{s_\ve(\beta)}}{1-\frac{s_\ve(\alpha)}{s_\ve(\beta)}}}.
\]
By Lemma~\ref{lem:limits} we may assume $\mu(\beta)=\alpha$, so
\begin{align*}
s_\ve(\alpha) &= \int_0^\alpha e^{-B(\mu^{-1}(x))/\ve^2}\,dx = \int_0^{\beta}e^{-B(x)/\ve^2}\mu'(x)\,dx.
\end{align*}
Thus,
\[
\frac{s_\ve(\alpha)}{s_\ve(\beta)} = \frac{1}{\bar\nu_\ve}\int_0^\beta \nu_\ve(x)\mu'(x)\,dx
\]
where
\[
\nu_\ve(x) = e^{-B(x)/\ve^2}, \qquad \bar\nu_\ve = \int_0^\beta e^{-B(z)/\ve^2}\,dz.
\]
From Lemma \ref{lem:approxidentity2} with $f(x)\coloneqq \mu'(x)$ it now follows that $p_\ve$ converges if either $0$ is a Lebesgue point for $\mu'$, or $\lim_{x\to0}\mu'(x)={-\infty}$.
In the former case, we notice that $0$ is a Lebesgue point for $\mu'$ if
 the following limit exists:
\[
{\lim_{h\downarrow 0}}\frac{\int_0^h \mu'(z) \,dz}{h}=
{\lim_{h\downarrow 0}}\frac{ \mu(h) -\mu(0)}{h}.
\]
The right hand side of the last equation is the usual definition of the derivative.

To prove \eqref{eq:limit_prob2} notice that
\[
\lim_{h\downarrow 0}\frac{ \mu(h) -\mu(0)}{h}= \lim_{h\downarrow 0}\frac{ \mu(h)}{h}=
 \lim_{h\downarrow 0}\frac{B_-^{-1}\circ B_+(h)}{h}
=\lim_{u\downarrow 0}\frac{B_-^{-1}(u)}{B_+^{-1}(u)}.
\]
\end{proof}

\subsection{Repulsive, regularly varying drifts}
Although Theorem \ref{thm:ZeroNoiseRepulsive} provides an explicit expression \eqref{eq:limit_prob} of the limit probabilities, the limit \eqref{eq:limit_prob2} might be difficult to evaluate in practice. It is clearly easier to study existence of the limits
\begin{equation}\label{eq:equiv_a_B}
\lim_{x\downarrow0} \frac{a(-x)}{a(x)}
\end{equation}
or
\begin{equation}\label{eq:equiv_a_B1}
\lim_{x\downarrow0} \frac{B(-x)}{B(x)}
\end{equation}
than that for the inverse functions in \eqref{eq:limit_prob}. We will show that the limit  in \eqref{eq:limit_prob} can easily be calculated using \eqref{eq:equiv_a_B} or \eqref{eq:equiv_a_B1} if $a$ or $B$ are regularly varying at $0$.

Recall that a positive, measurable function $f\colon [0,\infty)\to(0,\infty)$ is \emph{regularly varying} of index $\gamma$ at $+\infty$ if $\lim_{x\to\infty}\frac{f(\lambda x)}{f(x)}=\lambda^\rho$ for all $\lambda>0$. It is regularly varying of index $\rho$ at $0$ if the function $x\mapsto f(1/x)$ is a regularly varying function of index $-\rho$ at $+\infty$. The set of regularly varying functions of index $\rho$
(at $+\infty$) is denoted by $R_\rho.$ It is well known that if $f\in R_\rho$, then
$f(x)=x^\rho \ell(x)$ for some slowly varying function $\ell$, i.e.~some $\ell\from[0,\infty)\to(0,\infty)$ for which $\lim_{x\to\infty}\frac{\ell(\lambda x)}{\ell(x)}=1$ for all $\lambda>0$.

We first consider the case when $B$ is regularly varying, and then the case when $a$ is. Note that the latter implies the former, but not {vice versa}.

\begin{proposition}\label{prop:B_regvar}
Assume that the functions $x\mapsto B_\pm(\pm x)$ are regularly varying of index
 $\rho>0$ at 0, and that the limit $c\coloneqq \lim_{x\downarrow0} B_-(-x)/B_+(x)$ exists
(or equals $\infty$). Then $\{p_\ve\}_{\ve>0}$ converges, and
\begin{equation}\label{key}
p \coloneqq \lim_{\ve\to0} p_{\ve} = \frac{{c^{-1/\rho}}}{1+c^{-1/\rho}}.
\end{equation}
If the functions $x\mapsto B_\pm(\pm x)$ are regularly varying of different indices $\rho_\pm,$
then
\[
p \coloneqq \lim_{\ve\to0} p_{\ve} =
\begin{cases}
1 & \rho_+<\rho_-\\
0 &  \rho_+>\rho_-.
\end{cases}
\]

\end{proposition}
\begin{proof}
It follows from \cite[Exercise~14, p.~190]{BTG} that if $f_1, f_2\colon(0,\infty)\to(0,\infty)$ are non-decreasing, regularly varying functions at $0$ of index $\rho>0$, then
\[
\lim_{x\to0}\frac{f_1(x)}{f_2(x)}=1 \qquad \text{if and only if} \qquad \lim_{x\to0}\frac{f_1^{-1}(x)}{f_2^{-1}(x)}=1,
\]
where $f_1^{-1}, f_2^{-1}$ are inverse functions. Write now $f_1(x)=B_-(-x)$, $f_2(x)=c B_+(x)$. Then
 \begin{equation}\label{eq:ratiolimit}
\lim_{x\to 0}\frac{f_1(x)}{f_2(x)}=1.
\end{equation}
The inverse function for $x\mapsto c B_+(x)$ is $x\mapsto B_+^{-1}(x/c),$ and $B_+^{-1}$ is
regularly varying of index $1/\rho$ (see \cite[Theorem 1.5.12]{BTG}), so
\[
B_+^{-1}(x/c)= (x/c)^{1/\rho}\ell_1(x/c)\sim (x/c)^{1/\rho}\ell_1(x)= c^{-1/\rho} B_+^{-1}(x)\qquad \text{as } x\to 0
\]
(where equivalence is meant in the sense of slowly
 varying functions). Hence, \eqref{eq:ratiolimit} yields
\begin{equation}\label{eq:ratiolimitinv}
\lim_{x\to0}\frac{B_-^{-1}(x)}{B_+^{-1}(x)}=c^{-1/\rho}.
\end{equation}
The same computation can be easily performed in reverse, so
\eqref{eq:ratiolimit} and \eqref{eq:ratiolimitinv} are equivalent, and
 the result now follows from Theorem \ref{thm:limitprobs} if $B_\pm$ are of the same index.

If $x\mapsto B_\pm(\pm x)$ are regularly varying of different indices $ {\rho_\pm},$ then
 the inverse functions are regularly varying functions of indices $\frac{1}{\rho_\pm},$ and the result is obvious.
\end{proof}

\begin{proposition}\label{prop:a_regvar}
Assume that both $x\mapsto a(\pm x)$ (for $x\geq0$) are regularly varying at $0$ with index $\rho>0$, and that the limit $c\coloneqq \lim_{x\downarrow0} \frac{-a(-x)}{a(x)}$ exists. Then $\{p_\ve\}_{\ve>0}$ converges, and
\[
p\coloneqq \lim_{\ve\to0}p_\ve = \frac{{c^{-1/(1+\rho)}}}{1+c^{-1/(1+\rho)}}.
\]
If the functions $x\mapsto a(\pm x)$ are regularly varying of different indices $\rho_\pm,$
then
\[
p \coloneqq \lim_{\ve\to0} p_{\ve} =
\begin{cases}
1 & \text{if } \rho_+<\rho_-\\
0 & \text{if }  \rho_+>\rho_-.
\end{cases}
\]
\end{proposition}
\begin{proof}
It follows from the Karamata theorem, see \cite[Theorem 1.6.1]{BTG}, that for $x>0$,
\begin{align*}
B(x)&= 2\int_0^x a(y) \, dy= 2\int_{1/x}^\infty  a(1/z)z^{-2} \, dz
= 2\int_{1/x}^\infty \ell(z)z^{-2-\rho} \, dz& \\
&\sim \frac{2 \ell(1/x) x^{1+\rho}}{1+\rho} \sim \frac{x a(x)}{1+\rho}
\end{align*}
as $x\to0$, and likewise for $x<0$. Thus, $x\mapsto B(\pm x)$ are regularly varying of index $1+\rho$. Letting
\[
c\coloneqq\lim_{x\downarrow0}\frac{-a(-x)}{a(x)},
\]
we can now apply Proposition \ref{prop:B_regvar} with $1+\rho$ in place of $\rho$ and get the desired result.

The case when $a(\pm x)$ are regularly varying with different indices can be considered similarly,
 cf. Proposition \ref{prop:B_regvar}.
\end{proof}

Finally, we provide a result which simplifies the computation of the limit distribution for severely oscillating drifts.
\begin{proposition}\label{prop:oscillatingdrift}
Let $a\from\mbR\to\mbR$ satisfy $xa(x)\geq0$ for all $x\in\mbR$, and assume that it is of the form
\[
a(x) = b(x) + |x|^\gamma g(\tfrac1x),
\]
where $\gamma>0$, $b$ is regularly varying at $0$ of order  $\rho<\gamma+1$, and $g\in L^\infty(\mbR)$ is such that its antiderivative $G(x)=\int_0^x g(y)\,dy$ also lies in $L^\infty(\mbR)$. Assume also that the limit $c\coloneqq \lim_{x\downarrow0} \frac{-b(-x)}{b(x)}$ exists. Then $\{p_\ve\}_{\ve>0}$ converges, and
\[
p\coloneqq \lim_{\ve\to0}p_\ve = \frac{{c^{1/(1+\rho)}}}{1+c^{1/(1+\rho)}}.
\]
\end{proposition}

\begin{proof}
We claim first that
\[
\int_0^x y^\gamma g(\tfrac1y)\,dy = O(x^{1+\gamma}) = o(x^\rho) \qquad \text{as } x\downarrow0.
\]
Indeed,
\begin{align*}
\biggl|\int_0^x y^\gamma g(\tfrac1y)\, dy\biggr| &= \biggl|\int_{1/x}^\infty z^{-2-\gamma}  g(z)\, dz\biggr| \\
&= \biggl|-x^{2+\gamma}G(\tfrac1x) - (\rho+2)\int_{1/x}^\infty z^{-3-\gamma}G(z)\,dz\biggr| \\
&\leq x^{2+\gamma}\|G\|_{L^\infty} + (\rho+2)\|G\|_{L^\infty}\int_{1/x}^\infty z^{-3-\gamma}\,dz \\
&= \Bigl(1 + \frac{\rho+2}{\gamma+2}\Bigr)x^{2+\gamma}\|G\|_{L^\infty}.
\end{align*}
It follows that the antiderivative $B(x)=\int_0^x a(y)\,dy$ equals a regularly
 varying function of order $1+\rho$, plus a term of order $o(x^{1+\rho})$. Following the
same procedure as in the proof of Proposition \ref{prop:a_regvar} yields the desired result.
\end{proof}

\section{Proof of Theorem \ref{thm:ZeroNoisePositiveDrift111}}\label{section:finalOfTheorem1.1}
We have already proven the Theorem in Section \ref{sec:positive_drift} if $a>0$ a.e.~in a small neighborhood of 0.
We will only prove the result for $a$ such that
$a\geq 0$ for negative $x$ and $\int_0^{R}\frac{dy}{a(y)\vee 0}<\infty$ for all $R>0$. The general
case, i.e., $a(x)\geq 0$ for a.e.~$x\in(-\delta_0,0)$ and $\int_0^{\delta_0}\frac{dy}{a(y)}<0$, is considered similarly to the reasoning in Section \ref{sec:positive_drift}.

{It follows from the comparison theorem that for any $x>0$ we have the inequality
$X_\ve(t)\leq X_\ve^x(t)$ for $t\geq 0$ with probability 1,
where $X_\ve^x$ is a solution of \eqref{eq:ode_pert} that started from $x$, $X_\ve^x(0)=x.$
Since $a$ is a.e.~positive on $(0,x)$, we have already seen that   $\{X_\ve^x(t)\}_{\ve}$ converges to
$\psi_+(\psi_+^{-1}(x)+t)$ as $t\to\infty.$ Thus, any limit point of $\{X_\ve(t)\}_\ve$ must be less than or equal to $\psi_+(\psi_+^{-1}(x)+t)$ for any $x>0$, almost surely.
Therefore, any limit point of $\{X_\ve(t)\}_\ve$ does not exceed $\psi_+(t).$}

Define the function
\[
a_n(x):=\begin{cases}
a(x) & \text{if } x\geq 0 \\
-\tfrac1n a(-\tfrac{x}{n}) & \text{if } x<0,
\end{cases}
\]
and denote the corresponding solutions to stochastic differential equations by $X_{n,\ve}(t).$
Let us apply Theorem \ref{thm:limitprobs} to the sequence $\{X_{n,\ve}\}_\ve.$
Calculate the limit \eqref{eq:limit_prob2}:
\[
B_{n,+}(x)=\int_0^x a(y) dy,\qquad B_{n,-}(x)=\int_0^x a(y/n) dy/n=B_+(x/n).
\]
Thus,
\[
(B_{n,-})^{-1}(u)=n (B_{n,+})^{-1}(u), \qquad \lim_{u\downarrow0}\frac{(B_{n,-})^{-1}(u)}{(B_{n,+})^{-1}(u)}=n,
\]
and we get convergence
\[
P_{X_{n,\ve}}\Rightarrow \frac{1}{n+1}\delta_{-n\psi_+(n^{-2}t)}+\frac{n}{n+1}\delta_{\psi_+(t)} \qquad \text{as }\ve\to0.
\]
By the comparison theorem we have the inequality $X_{n,\ve}(t)\leq X_\ve(t), t\geq 0$ with probability 1. Therefore, any limit point of $\{X_\ve\}_\ve$ equals $\psi_+$ with probability at least $\frac{n}{n+1}.$
We conclude that the limit of $\{X_\ve\}_\ve$ exists and equals $\psi_+$ almost surely.
The limit is non-random, so we have  convergence in probability, as in \eqref{eq:C2}. This finishes the proof of Theorem \ref{thm:ZeroNoisePositiveDrift111}.

\section{Examples}\label{sec:examples}

\begin{example}\label{ex:1}
For some fixed $\rho\in (0,1)$ we consider the function
\[
a(x)\coloneqq \sign(x)|x|^\rho \bigl(1+ \tfrac{1}{2}\phi\bigl(\tfrac{1}{x}\bigr)\bigr) \qquad \text{where } \phi(y)\coloneqq\sum_{n\in\mbZ}\ind_{[2n-1,2n)} - \ind_{[2n,2n+1)},
\]
defined for all $x\neq0$. Using Proposition \ref{prop:oscillatingdrift} with $b(x)=\sign(x)|x|^\rho$, $\gamma=\rho$ and $g(y)=\tfrac12\sign(y)\phi(y)$, we get $c=1$, and that $p_\ve \to \frac{1}{2}$. We also see that $a$ satisfies the repulsive condition \eqref{eq:osgoodrepulsive} of Theorem \ref{thm:ZeroNoiseRepulsive}, so we conclude that
\[
P_\ve \Rightarrow \tfrac12\delta_{\psi_-} + \tfrac12\delta_{\psi_+} \qquad\text{as } \ve\to0
\]
where $\psi_\pm$ are the maximal classical solutions.

Figure \ref{fig:Example51} shows an ensemble of approximate solutions for the above drift. We used noise sizes \(\varepsilon = \frac{3^{-i}}{e} ,\ i=-2, \dots, -9\), and computed 150 samples of the solution with the Euler--Maruyama scheme with a step size $\Delta t=2.5\times 10^{-3}$ up to time $t=0.5$. The left-hand figure shows all sample paths (vertical axis) as a function of time (horizontal axis), where bigger \(\varepsilon\) were given a lighter shades of grey. The sample paths with the smallest \( \varepsilon \) are depicted in red. The right-hand figure shows the cumulative distribution function of the samples at the final time \( t = 0.5 \) using the smallest value for \( \varepsilon \). We can clearly see that the solution is concentrated on the extreme sample paths $\psi_-,\psi_+$, each with probability $\tfrac12$.

\begin{figure}
	\includegraphics[width=0.49\linewidth]{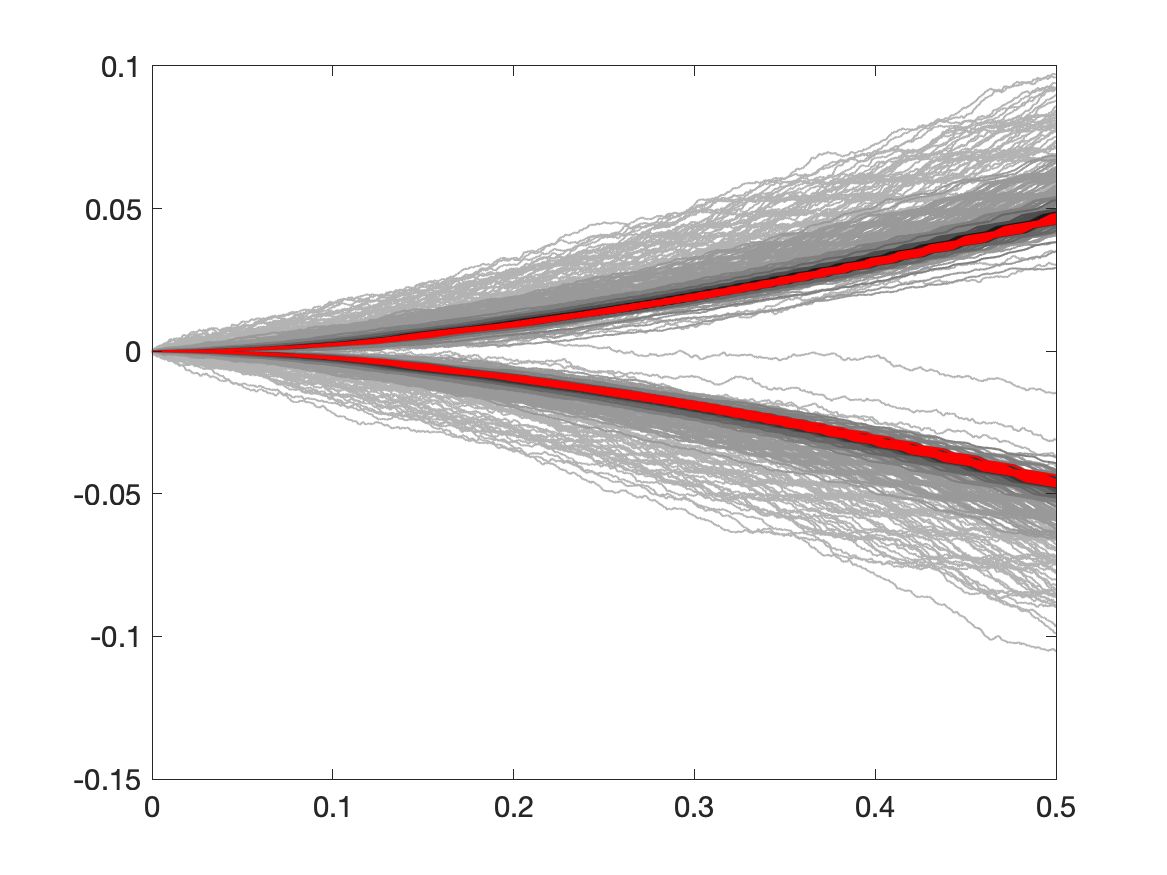}
	\includegraphics[width=0.49\linewidth]{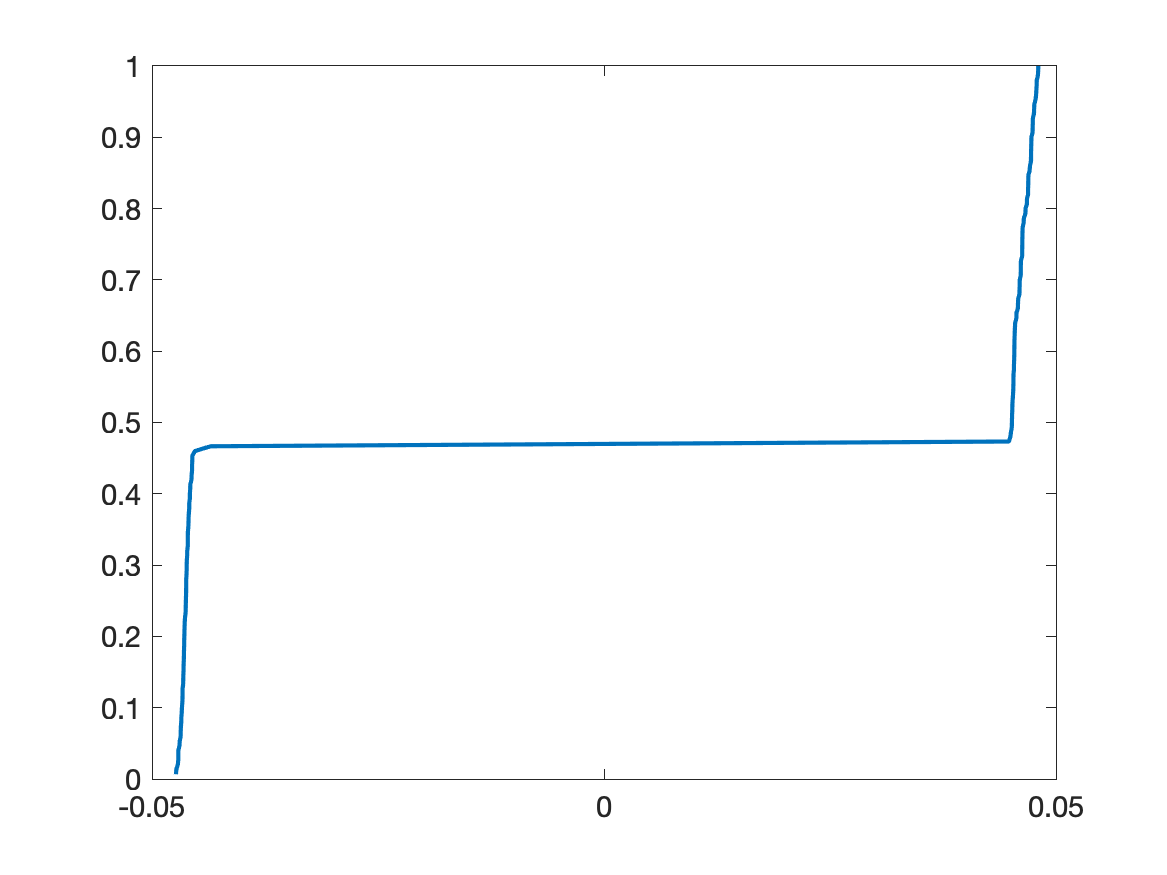}
	\caption{Sample paths (left) and cumulative distribution function (right) for Example \ref{ex:1}.}\label{fig:Example51}
\end{figure}

\end{example}

\begin{example}\label{ex:2}
Let $a(x)=x^\beta$, $x>0,$ where $\beta\in (0,1).$ We claim that we can continuously extend $a$ to the set $(-\infty,0]$ such that

\begin{enumerate}[label=(\alph*)]
\item $-a(-x)\leq a(x)<0$ for all $x<0$;

\item $\int_{-1}^0 \frac{1}{a(x)}dx=-\infty,$ i.e., the Osgood condition is not satisfied to the left of zero;

\item $P_{X_\ve}\Rightarrow \tfrac12 \delta_{\psi_+}+\tfrac12 \delta_{0}$ as $\ve\to0,$ i.e., the limit process with probability $\tfrac12$ moves like the maximal positive solution $\psi_+(t)=((1-\beta)t)^{\frac{1}{1-\beta}}, t\geq 0$
and stays at 0 forever with probability $\tfrac12$ too.
\end{enumerate}
This example is not covered by the theory in the previous sections, and should therefore be read as a demonstration of the complex behaviours that can occur in the zero noise limit. Note also that the zero-noise limit is \textit{not} only concentrated on the maximal solution $\psi_+$, but also on the trivial solution $\psi_-\equiv0$.

Before we construct the extension, let us provide some simple preliminary analysis.
If a function $a\from\mbR\to \mbR$ satisfies the linear growth condition, then the family $\{X_\ve\}_\ve$
is weakly relatively compact. If additionally the function $a$ is continuous, then
any limit point of  $\{X_\ve\}_\ve$ satisfies \eqref{eq:ode}. Both conditions
(a) and (b) yield that any solution to \eqref{eq:ode}, and hence any limit point of
$\{X_\ve\}_\ve$ has a form
\begin{equation}\label{eq:limit_sol}
X_0(t)=
\begin{cases}
0, & t\leq \tau\\
((1-\beta)(t-\tau))^{\frac{1}{1-\beta}},& t> \tau,
\end{cases}
\end{equation}
where $\tau\in[0,\infty].$ Our aim is to find an extension of $a$ such that
\begin{equation}\label{eq:tau_probab}
\Pr(\tau=0)= \Pr(\tau=\infty)=\tfrac12
\end{equation}
 for any limit point
$X_0$ having representation
\eqref{eq:limit_sol}.

Let $A=\cup_{k\geq 1}[-\frac{1}{2^k}, -\frac{1}{2^k}+\frac{1}{4^k}].$ Set
\begin{align*}
\tilde a(x) &\coloneqq \sign(x) a(|x|)\ind_{x\notin A}=
\begin{cases}
x^\beta, & x> 0\\
 -|x|^\beta,& x\leq 0,\ x\notin A \\
0,& x\leq 0,\ x\in A,
\end{cases}  \\
\bar a(x) &\coloneqq \sign(x) a(|x|)=  \sign(x) |x|^\beta=
\begin{cases}
x^\beta, & x> 0\\
 -|x|^\beta ,& x\leq 0.
\end{cases}
\end{align*}
Define $a$ on $(-\infty,0)$ to be any negative, continuous function such that $\int_{-\delta}^0 \frac{1}{a(x)}dx=-\infty$ for any $\delta>0$, and
\[
\bar  a(x)\leq a(x) \leq \tilde  a(x) \text{ for all } x\in(-\infty,0).
\]
It is clear that there exists a function $a$ satisfying these properties.
Introduce the transformed process
\[
Y_\ve(t)\coloneqq \ve^{\frac{-2}{1+\beta}}  X_\ve\bigl(\ve^{\frac{2(1-\beta)}{1+\beta}}t\bigr).
\]
It can be seen (see \cite{PilipenkoProske2018} for a more general case) that
\begin{equation}\label{eq:1525}
d   Y_\ve(t)=  a_\ve( Y_\ve(t))dt + d  w_\ve(t),
\end{equation}
where $w_\ve(t)= \ve^{\frac{-(1-\beta)}{1+\beta}} w\bigl(\ve^{\frac{2(1-\beta)}{1+\beta}}t\bigr)$ is a Wiener process, and
\begin{equation}
\label{eq:1526}
  a_\ve(y)= \ve^{\frac{-2\beta}{1+\beta}}a\bigl(\ve^{\frac{ 2}{1+\beta}}y\bigr).
\end{equation}
Notice that $a_\ve(y)=a(y)$ for all $y\in (0,\infty)$ and for all $y<0$ such that $\ve^{\frac{ 2}{1+\beta}}y\notin A$. For all other $y<0$ we have the inequality $-|y|^\beta\leq a(y)<0$, by the choice of the function $a$.
We have convergence  $ a_\ve(y)\to \bar a(y)=\sign(y) |y|^\beta$ in Lebesgue measure on any interval $y\in[-R,R]$.
Observe also that
\[
\int_0^x a_\ve(y) dy \geq \int_0^x \hat a(y) dy \qquad \forall \ve>0,\ \forall x<0
\]
where
\[
\hat a(x)=
\begin{cases}
0,\ & x\in [-\tfrac32\cdot 2^n, -2^n ] \text{ for some }  n\in\mbZ \\
-|x|^\beta & \text{otherwise.}
\end{cases}
\]
In particular, the last estimate yields
\[
\sup_{\ve\in(0,1]}\lim_{R\to+\infty}
  \int_{-\infty}^{-R} \exp\biggl(-2\int_0^x a_\ve(y)\,dy\biggr) dx =0.
\]

Set
\[
\sigma^X_{\ve}(p)\coloneqq\inf\{t\geq 0 : X_\ve(t)=p\},
\]
\[
\sigma^Y_{\ve}(p)\coloneqq\inf\{t\geq 0 : Y_\ve(t)=p\} .
\]
The observations above and formulas of Theorem \ref{thm:exit_time} yield that for any $R>0$, and for any sequences $\{R^\pm_\ve\}$ such that $\lim_{\ve\to0}R^\pm_\ve=\pm\infty$ we have
\begin{align*}
&\nqquad\lim_{\ve\to0}\Pr\Bigl(\sigma^Y_{\ve}(R) < \sigma^Y_{\ve}(-R) \mid  Y_\ve(0)=0\Bigr) =
\lim_{\ve\to0}\Pr\Big( \sigma^Y_{\ve}(-R)< \sigma^Y_{\ve}(R) \mid  Y_\ve(0)=0\Big) \\
={}&\lim_{\ve\to0}\Pr\Big( \sigma^Y_{\ve}(R^+_\ve)<  \sigma^Y_{\ve}(R^-_\ve) \mid  Y_\ve(0)=0\Big)
= \lim_{\ve\to0}\Pr\Big( \sigma^Y_{\ve}(R^-_\ve)< \sigma^Y_{\ve}(R^+_\ve) \mid  Y_\ve(0)=0\Big)\\
={}&\tfrac12.
\end{align*}
Hence, for any $\delta^\pm>0$ we have
\begin{equation}\begin{aligned}\label{eq:1566}
&\nqquad\lim_{\ve\to0}\Pr\Big(  \sigma^X_{\ve}\big(R \ve^{\frac{2}{1+\beta}}\big)
< \sigma^X_{\ve}\bigl(-R \ve^{\frac{2}{1+\beta}}\big) \mid  X_\ve(0)=0\Big)\\
={}&
\lim_{\ve\to0}\Pr\Big(  \sigma^X_{\ve}\bigl(-R \ve^{\frac{2}{1+\beta}}\big)<
 \sigma^X_{\ve}\big(R \ve^{\frac{2}{1+\beta}}\big) \mid  X_\ve(0)=0\Big) \\
={}&\lim_{\ve\to0}\Pr\Big(\sigma^X_{\ve}(\delta^+)< \sigma^X_{\ve}(\delta^-) \mid  X_\ve(0)=0\Big)
= \lim_{\ve\to0}\Pr\Big(\sigma^X_{\ve}(\delta^-)< \sigma^X_{\ve}(\delta^+) \mid X_\ve(0)=0\Big)\\
={}&\tfrac12.
\end{aligned}
\end{equation}
  Hence, if $X_0$ is a limit point of $\{X_\ve\}$ having representation
  \eqref{eq:limit_sol}, then $\Pr(\tau=\infty)\geq \tfrac12.$
It also follows from  Theorem \ref{thm:exit_time} that for any $R>0$
\[
\sup_{\ve>0}\Exp \Big(\sigma^Y_{\ve}(R)\wedge \sigma^Y_{\ve}(-R) \mid Y_\ve(0)=0\Big)<\infty.
\]
Thus,
\begin{equation}\label{eq:1575}
\sigma^X_{\ve}\big(R \ve^{\frac{2}{1+\beta}}\big)\wedge \sigma^X_{\ve}\bigl(-R \ve^{\frac{2}{1+\beta}}\big)
\overset{\Pr}\to 0 \qquad \text{as } \ve\to0
\end{equation}
if $X_{\ve}(0)=0.$

Let $\bar X_\ve $ be a solution to
\[
d \bar X_\ve(t) =\bar a\big(\bar X_\ve(t)\big)dt +\ve d w(t)
\]
and define
\[
\bar Y_\ve(t)\coloneqq \ve^{\frac{-2}{1+\beta}}\bar X_\ve\bigl(\ve^{\frac{2(1-\beta)}{1+\beta}}t\bigr).
\]
Then (cf.~\eqref{eq:1525}, \eqref{eq:1526})
\[
 d \bar Y_\ve(t)= \bar a(\bar Y_\ve(t))dt + d w_\ve(t).
\]
In particular, if $\bar X_\ve(0)= R \ve^{\frac{2}{1+\beta}}$ for all $\ve>0,$ where $R$ is a constant, then all processes $\bar Y_\ve$ have the same distribution independent of $\ve.$

Notice that for any $R>0$,
\begin{equation}\label{eq:1597}
 \Pr\Bigl(X_\ve(t) = \bar X_\ve(t), \ t\in [0, \sigma^X_{\ve}(0) ]  \mid X_\ve(0)= \bar X_\ve(0)= R \ve^{\frac{2}{1+\beta}}\Bigr) =1.
\end{equation}
 and
\begin{equation}\begin{aligned}\label{eq:1600}
p_R &\coloneqq  \Pr\Bigl(\sigma^X_{\ve}(0)=\infty \mid X_\ve(0)= R \ve^{\frac{2}{1+\beta}}\Bigr) \\
&= \Pr\Bigl(X_\ve(t)>0, t\geq 0 \mid X_\ve(0)= R \ve^{\frac{2}{1+\beta}}\Bigr)
\\
&=\Pr\Bigl(\bar X_\ve(t)>0, t\geq 0 \mid \bar X_\ve(0)= R \ve^{\frac{2}{1+\beta}}\Bigr) \\
&= \Pr\Bigl(\bar Y_\ve(t)>0 \mid \bar Y_\ve(0)= R\Bigr)\to 1 \qquad \text{as } R\to\infty.
\end{aligned}\end{equation}
It follows from \cite{PilipenkoProske2018}  that if $\bar X_\ve(0)= R \ve^{\frac{2}{1+\beta}}, \ve>0$,
then
\begin{equation}
\label{eq:1611}
\bar X_\ve \Rightarrow p_R\delta_{\psi_+}+(1-p_R)\delta_{\psi_-} \qquad \text{as } \ve\to0.
\end{equation}
Hence, \eqref{eq:1566}, \eqref{eq:1575}, \eqref{eq:1597}, \eqref{eq:1600},
   and \eqref{eq:1611} yield that for any limit point $X_0$
of $\{X_\ve\}$ we have $\Pr(\tau=0)\geq \tfrac12.$ This concludes the proof of the convergence $P_{X_\ve}\Rightarrow \tfrac12 \delta_{\psi_+}+\tfrac12 \delta_{0}$ as $\ve\to0.$

Figure \ref{fig:Example52} shows the same type of simulation as in Example \ref{ex:1}. From the figure it is clear that for small $\ve$, the samples split in two groups of equal size, one moving along $\psi_+$ and the other remaining around the origin. As the noise decreases, the left-going samples concentrate around the trivial solution $X\equiv0$.
\begin{figure}
	\includegraphics[width=0.49\linewidth]{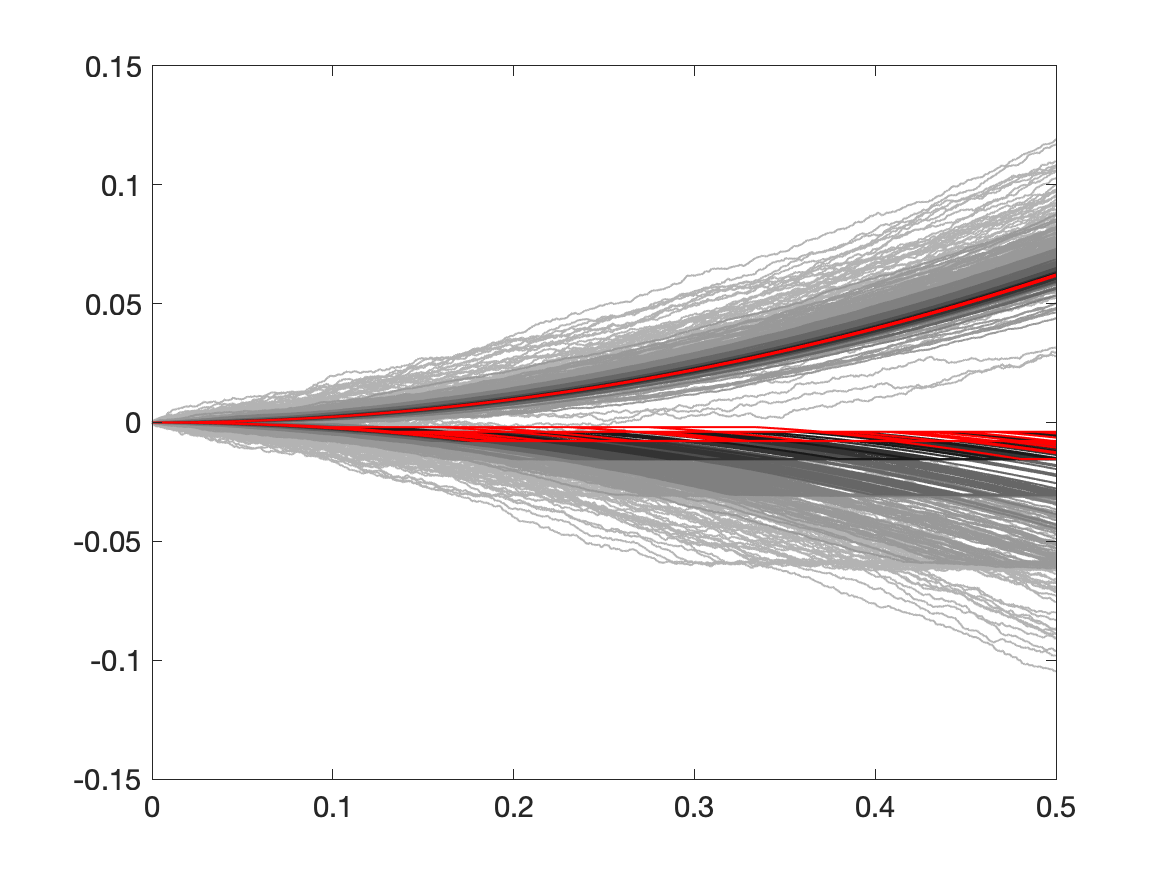}
	\includegraphics[width=0.49\linewidth]{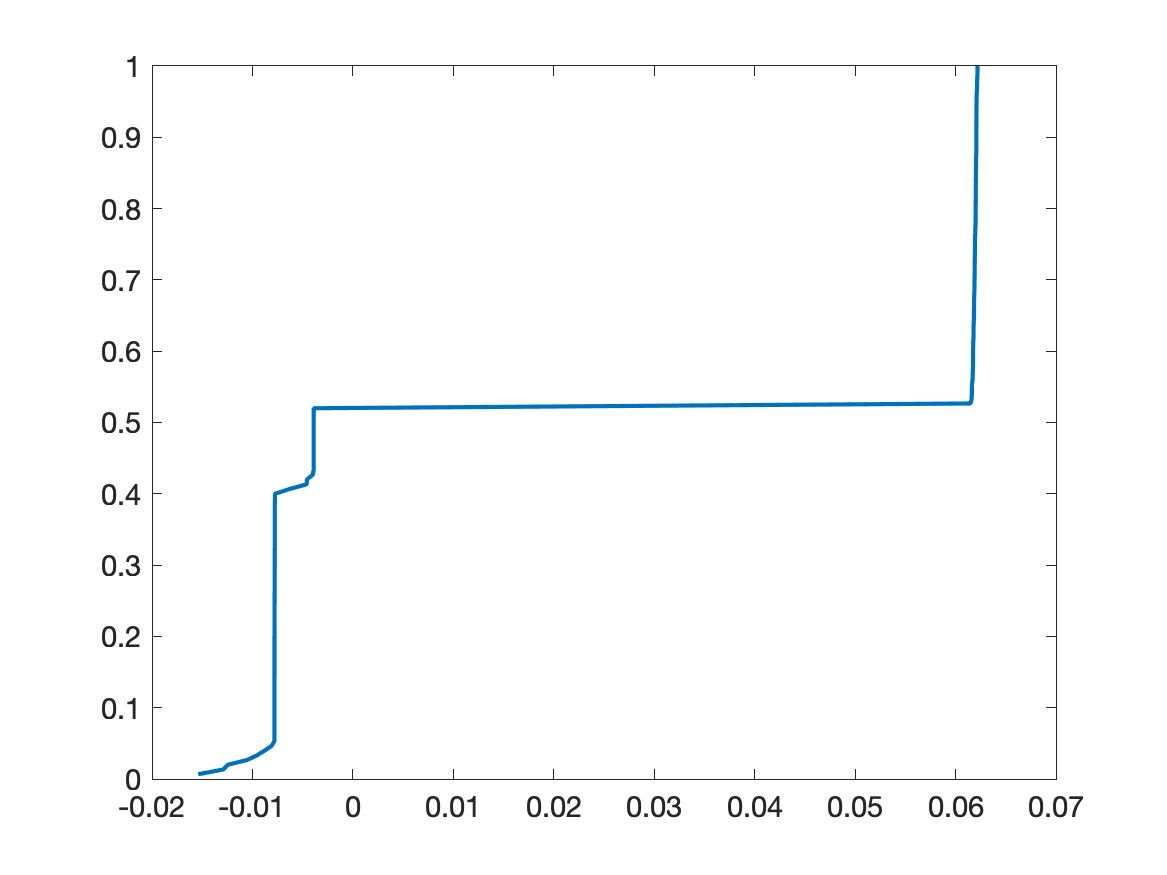}
	\caption{Sample paths (left) and cumulative distribution function (right) for Example \ref{ex:2}.}\label{fig:Example52}
\end{figure}
\end{example}

\appendix
\section{Appendix}\label{app:comparisonprinciple}
\begin{proof}[Proof of Theorem \ref{thm:comparisonThm}]
For a sequence of numbers $0<\ve_n\to0$, let $a_{i,n} = a_i*\omega_{\ve_n}$, where $\omega_\ve(z) = \ve^{-1}\omega\big(z\ve^{-1}\big)$ and $\omega\in C_c^\infty(\mbR)$ is a nonnegative mollifier. Let $X_{i,n}$ be the unique solution of
\begin{equation}
dX_{i,n} = \drift_{i,n}( X_{i,n}) dt + dW, \qquad i = 1,2,\ n\in\mbN. \label{eq:comparison}
\end{equation}
For the smoothened drift functions it still holds \( \drift_{1,n} \leq \drift_{2,n}\). Therefore, it follows from the classic comparison theorem that $X_{1,n} \leq X_{2,n}$ (see e.g.~the comparison theorem in \cite{IkedaWatanabe1981}).

The application  of Theorem \ref{thm:convergenceSDE_Thm} completes the proof in the case when $a_1, a_2\in L^\infty(\mbR)$. If $a_1, a_2$ are only locally bounded, then we
approximate $X_1,X_2$  by solutions to SDEs with drifts $a_{i,M}\coloneqq a_i\ind_{[-M,M]}$. It follows from \cite[Remark 3b, p.~145]{Zvonkin1974} that
\[
\Pr\bigl(X_i(t)=X_{i,M}(t) \;\forall\ t\leq \tau_{i,M}\bigr)=1, \qquad i=1,2,
\]
where $\tau_{i,M}= \inf\{t\geq 0 : |X_i(t)|\geq M\}.$ We have already proved that $X_{1,M}(t)\leq X_{2,M}(t)$ almost surely. This completes the proof of the theorem.
\end{proof}

\begin{proof}[Proof of Lemma \ref{lem:approxidentity}]
We assume that $f$ is positive; the negative case follows similarly.
Denote $B(z)\coloneqq\int_{\alpha}^z f(u)\,du$. Then $B$ is absolutely continuous and invertible, and since $B'(z)=f(z)>0$ for a.e.~$z$, the inverse $B^{-1}$ is also absolutely continuous (see e.g.~\cite[Exercise 5.8.52]{Bogachev2007}). Hence, we can write
\begin{align*}
g_\ve(y) &= \int_{y}^\beta e^{-(B(z)-B(y))/\ve^2}\frac{B'(z)}{\ve^2} g(z)\,dz \\
&= \int_{B(y)}^{B(\beta)} \frac{e^{-(v-B(y))/\ve^2}}{\ve^2}g\big(B^{-1}(v)\big)\,dv
\end{align*}
(where we made the change of variables $v=B(z)$). The function $[0,\infty)\ni v \mapsto \frac{e^{-v/\ve^2}}{\ve^2}$ is an approximate identity and therefore
\[
g_\ve(y) \to g(B^{-1}(B(y))) = g(y) \qquad \text{as } \ve \to 0
\]
in $L^1((\alpha,\beta))$, and pointwise whenever $v=B(y)$ is a Lebesgue point for $v \mapsto g\big(B^{-1}(v)\big)$; see e.g.~\cite[Theorems 8.14, 8.15]{Folland1999}. But $B$ and $B^{-1}$ are absolutely continuous, so these points coincide with the Lebesgue points for $g$.
\end{proof}

\section*{Acknowledgements}
U.~S.~Fjordholm was partially supported by the Research Council of Norway project \textit{INICE}, project no.~301538. A.~Pilipenko acknowledges the support by the National Research Foundation of Ukraine (project 2020.02/0014 ``Asymptotic regimes of perturbed random walks: on the edge of modern and classical probability'') and the Senter for internasjonalisering av utdanning (SIU), within the project Norway--Ukrainian Cooperation in Mathematical Education, project number CPEA-LT-2016/10139.

%

\end{document}